\documentclass[12pt,a4paper]{amsart}
\usepackage[utf8]{inputenc}
\usepackage[greek,english,francais]{babel}
\usepackage[T1]{fontenc}
\usepackage{charter}
\usepackage{mathabx}
\usepackage{listings}

\usepackage{color}      
\usepackage{indentfirst}        
\usepackage{latexsym,bm}        
\usepackage{array}
\usepackage{extarrows}
\usepackage{subfiles}
\usepackage{amsxtra}
\usepackage{amstext}
\usepackage{xypic}
\usepackage{dsfont} 
\usepackage{amsmath,amsthm,amssymb,amscd,amsfonts,typearea}

\usepackage{calligra}
\usepackage{mathrsfs}
\usepackage{mathtools}

\usepackage{enumerate}
\usepackage{enumitem}
\usepackage{moreenum}

\usepackage[colorlinks,  linkcolor=black,
anchorcolor=blue,
citecolor=blue]{hyperref}
\usepackage{cleveref}

\usepackage[mathscr]{eucal}
\usepackage{mathrsfs}
\usepackage{pdfsync}

\usepackage[a4paper,right=2.5cm,left=2.5cm,
top=2.5cm,bottom=2cm]{geometry}

\newcommand{\Z}{\mathbb{Z}}
\newcommand{\C}{\mathbb{C}}
\newcommand{\R}{\mathbb{R}}
\newcommand{\bN}{\mathbb{N}}

\newcommand{\E}{\mathbb{E}} 

\DeclareMathOperator{\ric}{Ric}
\DeclareMathOperator{\ord}{ord}
\DeclareMathOperator{\Div}{Div}

\begin{document}
\selectlanguage{english}

\title[Large deviations on punctured Riemann surfaces]{
Large deviations for zeros of holomorphic 
sections on punctured Riemann surfaces}
\author{Alexander Drewitz, Bingxiao Liu and George Marinescu} 
\address{Universit{\"a}t zu K{\"o}ln,  Mathematisches Institut,
    Weyertal 86-90,   50931 K{\"o}ln, Germany}
    \email{drewitz@math.uni-koeln.de}
\email{bxliu@math.uni-koeln.de}
\email{gmarines@math.uni-koeln.de}
\thanks{The authors are partially supported by the 
DFG Priority Program 2265 \lq Random Geometric Systems\rq.}

\date{\today}

\maketitle

\begin{abstract}
In this article we obtain large deviation estimates for zeros 
of random holomorphic sections on punctured Riemann surfaces. 
These estimates are then employed to yield estimates for the respective hole 
probabilities. A particular case of relevance that is covered by our setting is that of cusp forms on arithmetic surfaces.
Most of the results we obtain also allow for reasonably general probability distributions on holomorphic 
sections, which shows the universal character of these estimates.
Finally, we also extend our results to the case of certain higher 
dimensional complete Hermitian manifolds, which are not necessarily 
assumed to be compact.

\end{abstract}

\theoremstyle{plain}
\newtheorem{theorem}{Theorem}[section]
\newtheorem{lemma}[theorem]{Lemma}
\newtheorem{proposition}[theorem]{Proposition}
\newtheorem{corollary}[theorem]{Corollary}

\theoremstyle{remark}
\newtheorem{definition}[theorem]{Definition}
\newtheorem{remark}[theorem]{Remark}
\newtheorem{example}[theorem]{Example}



\selectlanguage{english}
\tableofcontents

\setcounter{page}{1}
\numberwithin{equation}{section}

\section{Introduction}

\subsection{Zeros of random holomorphic sections}
One particularly important aspect in study of random functions or stochastic processes has been the investigation of their zero sets, see \cite{AT:07} and \cite{AW:09} as well as references therein. In order to obtain stronger implications we will impose some further assumptions here, focusing on the case of geometric generalizations of random polynomials, i.e.\ random holomorphic sections. 

In order to motivate and introduce our setting, we begin with recalling that for analytic functions
$\sum_{n=0}^\infty a_nz^n$
whose coefficients $a_n$ are assumed to be
independent random variables,
Offord proved 
in his fundamental article \cite{Off67}  the exponential
decay of the tail probabilities of an analytic function having an excess or
deficiency of zeros in a given region. 
More recently, Sodin \cite{So00} used Offord’s
method to improve Offord's exponential bound on
the probability that a
random analytic function has no zeros in a disk
of radius $r$ (hole probability), by showing that it decays at least at the rate 
$\mathcal{O}(e^{-Cr^2})$. This result has since been refined and 
extended in various ways in a series of papers \cite{Kr06,PV05,ST04,Zb07}.

A change of paradigm has then been introduced by Shiffman, Zelditch 
and Zrebiec in the seminal article \cite{SZZ} by 
 generalizing the situation described above to compact K\"ahler manifolds
and zeros of holomorphic sections $H^0(X,L^p)$ of powers
of a positive line bundle. The principal interest in this setting is the study of the distribution
of zeros as $p\to\infty$. In this situation, the power series representation of an analytic function is not canonical anymore, and, as a consequence, one has to replace the
arguments based on the power series by more analytic and geometric methods which are
appropriate for the study of holomorphic sections. 
In particular, these include tools such as Bergman kernels and coherent states 
asymptotics, 
which are by now deeply rooted
in the study of the geometry of K\"ahler manifolds.

In this paper, we generalize the results of \cite{SZZ} in two directions: 
From a geometric point of view, we are now concerned with 
the case of noncompact (complete) complex manifolds. 
On the probabilistic side we allow for 
probability measures which are no longer Gaussian anymore; 
instead, these probability measures will be assumed to fulfill some 
rather general conditions which entail a certain universality of the results we obtain.
In \cite{BCM20} (see also \cite{BCHM18} for a survey) 
it was shown that the equidistribution of zeros
takes place for a large class of probability measures satisfying
a certain moment condition (e.\,g.\ measures with
heavy tail probability and small ball probability, or measures with support contained in totally 
real subsets of the complex probability space).
Analogous equidistribution results for non-Gaussian ensembles are 
proved in \cite{B6,B7,BL15,DS06}.
In this paper we consider 
probability measures satisfying very mild conditions in terms
of their densities (see Section \ref{section1.3prob}) .

We will primarily focus on the case of a Riemann surface with cusps
and prove large deviations estimates for zeros
of random holomorphic sections of high powers $L^p$ of a holomorphic line bundle
$L$ whose curvature equals
the Poincar\'e metric near the cusps. A special case is that of
cusp forms of high degree $2p$.
For such a bundle $L$, Auvray, Ma and Marinescu \cite{AMM:20} 
(cf.\ also \cite{AMM:20b}) gave a very precise 
description of the Bergman kernel near the cusps; in particular, they provide
an optimal uniform estimate of the
supremum norm of the Bergman kernel, involving the fractional 
growth order $p^{3/2}$ in the
tensor power (the growth order is $p$ in the compact case).
Using this estimate we obtain in Theorem \ref{thm:1.3.2} asymptotic bounds for
the expectation and the tail probability of the maximum modulus of a 
random section on an open set. What is more, we also establish extensions to the case of higher dimensional Hermitian 
manifolds under suitable conditions.

We now introduce the setting for our bounds on the excess or deficiency probabilities of zeros. Indeed, for a compact K\"ahler manifold $(X,\omega)$ endowed with
a Hermitian holomorphic line bundle $(L,h)$
with positive curvature $\omega=c_1(L,h)$, Shiffman-Zelditch \cite{ShZ99}
showed that the normalized currents of integration $\frac1p[\Div(s_p)]$ over 
zero divisors of a random sequence of sections $s_p\in H^0(X,L^p)$
converge almost surely to $c_1(L,h)$ as $p\to\infty$.
This result was generalized to the
noncompact setting in \cite{DMS12} and to the setting of singular metrics
whose curvature is a K\"ahler current in \cite{BCM20,CM11,DMM}. 
It holds also in our present setting and implies that
the number of zeros (counted with multiplicity) of a random
section $s_p$ in an open set $U$ with negligible boundary
is asymptotically equal to $p$ times the area of $U$ in the
metric given by $c_1(L,h)$.
In Theorem \ref{thm:1.2.1} we prove this result in our setting 
and show that the probability
that a section $s_p$ has an excess or deficiency of zeros in $U$ 
(when centered around its typical value)
decreases at rate $\exp(-Cp^2)$, this being consistent with the 
decay obtained in Sodin \cite{So00} cited above.

\subsection{Geometric setting: punctured Riemann surfaces}
We will use the notation $\overline{\Sigma}$ to denote a compact Riemann surface 
and write $D=\{a_{1},\ldots, a_{N}\}\subset \overline{\Sigma}$ for a finite set. 
The induced punctured Riemann surface will be denoted by 
$\Sigma=\overline{\Sigma}\backslash D,$ and  $\omega_{\Sigma}$ 
will be a Hermitian form on $\Sigma$. 
We furthermore let $L$ be a 
holomorphic line bundle on $\overline{\Sigma}$, and denote by $h$ 
a singular Hermitian metric on $L$ satisfying the following properties:
\begin{enumerate} [label=(\greek*)]
\item \label{item:alpha} $h$ is smooth over $\Sigma$, and for all 
$j \in \{ 1, \ldots, N\}$ 
there is a trivialization of $L$ in the complex neighborhood $\overline{V}_{j}$ 
of $a_{j}$ in $\overline{\Sigma}$, with associated coordinate $z_{j}$ such that 
$|1|^{2}_{h}(z_{j}) = |\log(|z_{j}|^{2})|$.
\item \label{item:beta} There exists $\varepsilon_{0}>0$ such that the (smooth) curvature 
	$R^{L}$ of $h$ satisfies $iR^{L} \geq \varepsilon_{0} \omega_{\Sigma}$ 
	over $\Sigma$ and moreover, $iR^L =\omega_{\Sigma}$ on $V_j 
	:=\overline{V}_j\backslash\{a_{j}\}$; in particular, $\omega_{\Sigma} = 
	\omega_{\mathbb{D}^{*}}$ in the local coordinate $z_{j}$ on $V_{j}$ 
	and $(\Sigma,\omega_{\Sigma})$ is complete.
\end{enumerate}

Here, $\omega_{\mathbb{D}^{*}}$ denotes the Poincar\'{e} metric on 
the punctured unit disc $\mathbb{D}^{*}$, normalized as
\begin{equation}
	\omega_{\mathbb{D}^{*}}=\frac{idz\wedge 
	d\overline{z}}{|z|^{2}\log^{2}(|z|^{2})}\,\cdot
	\label{eq:1.1.1}
\end{equation}
Since $h$ is 
assumed to be a Hermitian metric, on the local chart $V_{j}$ as in 
Assumption \ref{item:alpha}, the coordinate $z_{j}$ 
has norm strictly less than $1$, so that the area (volume) of $V_{j}$ with 
respect to measure $\omega_{\Sigma}$ is finite.

Let 
$J\in\mathrm{End}(T\Sigma)$ denote the complex structure of 
$\Sigma$ and write $g^{T\Sigma}=\omega_{\Sigma}(\cdot,J\cdot)$ for
the complete Riemannian metric on $\Sigma$, so that the corresponding 
Riemannian volume element is exactly $\omega_{\Sigma}$. For 
$x\in\Sigma$ and $v\in T_{x}\Sigma$ we denote by $\Vert v\Vert$ the norm 
of $v$ with the metric $g^{T\Sigma}_{x}$. For $x,y\in 
\Sigma$, we write $\mathrm{dist}(x,y)$ for their Riemannian distance. 
Furthermore, for $x\in\Sigma$ we set
\begin{equation}
	a(x)=iR^{L}_{x}/\omega_{\Sigma,x}\geq\varepsilon_{0}>0.
\end{equation}

For $p\geq 1$, we denote by $h^{p}:=h^{\otimes p}$ the metric induced by 
$h$ on $L^{p}|_{\Sigma}$. We write $H^{0}(\Sigma, L^{p})$ for the space 
of holomorphic sections of $L^{p}$ on $\Sigma$ and 
$\mathcal{L}^{2}(\Sigma, L^{p})$ for the space of 
$\mathcal{L}^{2}$-sections of $L^{p}$ on $\Sigma$. Set
\begin{equation}
H^{0}_{(2)}(\Sigma, L^{p})
=H^{0}(\Sigma, L^{p})\cap\mathcal{L}^{2}(\Sigma, L^{p})
=\Big\{s\in H^{0}(\Sigma, L^{p})\;:\; 
\|s\|^{2}_{\mathcal{L}^{2}}:=\int_{\Sigma}|s|_{h^{p}}^{2}\, 
\omega_{\Sigma}<\infty\Big\},
\label{eq:1.1.2}
\end{equation}
which we tacitly assume to be endowed with the $\mathcal{L}^{2}$-metric. 
Then the sections in 
$H^{0}_{(2)}(\Sigma, L^{p})$ extend to holomorphic sections of 
$L^{p}$ over $\overline{\Sigma}$, i.e.
\begin{equation}
	H^{0}_{(2)}(\Sigma, L^{p})\subset H^{0}(\overline{\Sigma}, L^{p}).
\end{equation}
Moreover, for $p\geq 2$, elements in $H^{0}_{(2)}(\Sigma, L^{p})$ are 
exactly the sections in $H^{0}(\overline{\Sigma}, L^{p})$ vanishing on the 
puncture divisor $D$.

In the 
sequel, we write $c_{1}(L,h)$ for the first Chern form of $(L,h)$, i.e.
\begin{equation}
	c_{1}(L,h)=\frac{i}{2\pi}R^{L}.
	\label{eq:1.2.2paris}
\end{equation}
Hence, as a volume form on $\Sigma$ we have $c_{1}(L,h) \geq 
\frac{\varepsilon_{0}}{2\pi}\omega_{\Sigma}$ due to \ref{item:alpha}.
We furthermore set
\begin{equation}
	d_{p}=\dim H^{0}_{(2)}(\Sigma, L^{p})<\infty
\end{equation}
and denote by $\chi(\Sigma)$ the Euler characteristic number of the 
punctured Riemann surface $\Sigma.$ 
Then, as a consequence of the Riemann-Roch Theorem, we infer that
\begin{equation}
	d_{p}=p\deg(L)+\chi(\Sigma),
	\label{eq:dpL}
\end{equation}
where $\deg(L)=\int_{\Sigma}c_{1}(L,h)<\infty$.

Furthermore, we 
will denote the Schwartz kernel of the orthogonal 
projection from $\mathcal{L}^{2}(\Sigma,L^{p})$ onto 
$H^{0}_{(2)}(\Sigma, L^{p})$, called Bergman kernel, 
by $B_{p}(x,y)$ for $x, y\in \Sigma$. If $S^{p}_{j}$, 
$j=1,\ldots, d_{p}$ is an orthonormal basis of $H^{0}_{(2)}(\Sigma, 
L^{p})$ with respect to the $\mathcal{L}^{2}$ inner product, then 
\begin{equation}
	B_{p}(x,y)=\sum_{j=1}^{d_{p}}S^{p}_{j}(x)\otimes 
	S^{p,*}_{j}(y)\in L^{p}_{x}\otimes L^{p,*}_{y},
	\quad\text{for $x,y\in \Sigma$,}
\end{equation}
where the duality is defined by $h^{p}$. In particular, $B_{p}(x,x)$ 
is a positive function in $x\in\Sigma$.

\subsection{Probabilistic setting}
\label{section1.3prob}
For each $p\in \mathbb{N}$, we will endow $H^{0}_{(2)}(\Sigma, L^{p})$ 
with a probability measure $\Upsilon_{p}$ and hence obtain a 
sequence 
of probability spaces $(H^{0}_{(2)}(\Sigma, 
L^{p}),\Upsilon_{p})_{p\in\mathbb{N}}$. In order to construct the
sequence $\{\Upsilon_{p}\}_{p\in \mathbb{N}}$ we proceed as 
follows. For each $p\in\mathbb{N}$, we fix an orthonormal basis 
$O_{p}=\{S^{p}_{j}\}_{j=1}^{d_{p}}$ for $H^{0}_{(2)}(\Sigma, 
L^{p})$ with respect to the respective $\mathcal{L}^{2}$-inner products. 
We assume given a family 
of independent $\C$-valued random variables 
$\{\eta^{p}_{j}\}_{p\in\mathbb{N},\, 1\leq j\leq d_{p}}$ 
such that the following are satisfied:
\begin{itemize}
	\item uniformly bounded densities: each 
	$\eta^{p}_{j}$ admits a probability density 
	function (PDF) $f^{p}_{j}$ on $\C$ with respect to the standard 
	Lebesgue measure on $\C\simeq \R^{2}$, 
	and there exists a constant $M_{0}>0$ such that 
	all $p$ and all $1\leq j\leq d_{p}$, 
	\begin{equation}
		\sup_{z\in \C} f^{p}_{j}(z)\leq M_{0};
		\label{eq:5.1.2}
	\end{equation}
	\item uniform lower bound for variances: for each $p$, the 
	random variables
	$\eta^{p}_{j}$, $1\leq j\leq d_{p}$ are centered (i.e., 
	$\mathbb{E}[\eta^{p}_{j}]=0$) and have the same variance 
	$\sigma^{2}_{p}>0$ ($\sigma_{p}>0$). 
	Moreover, there exists $c_{0}>0$ such that for all $p$, 
	\begin{equation}
		c_{0}\leq \sigma^{2}_{p} < \infty;
		\label{eq:5.1.1}
	\end{equation}	
	\item moment bounds: there exists $C_{0}>0$ 
	such that for all $p,\,1\leq j\leq d_{p}$, we 
	have
	\begin{equation}
		\mathbb{E}[|\eta^{p}_{j}|^{d_{p}}]\leq C_{0}(d_{p})^{d_{p}}.		
		\label{eq:5.1.3}
	\end{equation}
\end{itemize}

\begin{remark}\label{rk1.3.1}
All of the above conditions are rather natural in avoiding degeneracies. 
Indeed, condition \eqref{eq:5.1.2} limits the concentration of 
$\eta^{p}_{j}$ in small areas of $\C,$ and  
	condition 
	\eqref{eq:5.1.3} avoids an overly fast growth of moments. 
	The conditions are relatively mild in that it is easily seen to 
	be verified for a wide range of distributions including 
	e.g.\ sub-Gaussian or exponential distributions. 
	
\end{remark}

For each $p\in \mathbb N,$ the orthonormal basis $O_{p}$ induces an identification 
$H^{0}_{(2)}(\Sigma, L^{p})\simeq \C^{d_{p}}$, where the section 
$s_{p}=\sum_{j=1}^{d_{p}}z_{j}S^{p}_{j}$ maps to the vector 
$(z_{1},\ldots, z_{d_{p}})\in\C^{d_{p}}$. Denoting by
$\mathrm{dVol}_{p}$ the standard Lebesgue measure on 
$\C^{d_{p}}\simeq \R^{2d_{p}},$ this naturally induces a probability measure $\Upsilon_{p}$ on 
$H^{0}_{(2)}(\Sigma, L^{p})$ via 
\begin{equation}
	\prod_{j=1}^{d_{p}}f^{p}_{j}(z_{j}) \, \mathrm{dVol}_{p}.
	\label{eq:5.1.4}
\end{equation}
For later use, we will abbreviate the respective density as
\begin{equation}
	f^{p}(z_{1},\ldots,z_{d_{p}})=\prod_{j=1}^{d_{p}}f^{p}_{j}(z_{j}).
	\label{eq:5.1.5}
\end{equation}
Hence, using the above identification, a random section in 
$(H^{0}_{(2)}(\Sigma, L^{p}),\Upsilon_{p})$  with distribution $\Upsilon_{p}$ can be written as
\begin{equation}
	s_{p}=\sum_{j=1}^{d_{p}}\eta^{p}_{j} S^{p}_{j}.
	\label{eq:5.1.5bis}
\end{equation}

In general, $\Upsilon_{p}$ (and $f^{p}$) depends on both 
the choice of orthonormal basis 
$O_{p}$ and the sequence $\{\eta^{p}_{j}\}_{j=1}^{d_{p}}$.

In 
the sequel, we fix -- once and for all -- a choice of the above constants 
$M_{0}$, $c_{0}$, $C_{0}$. Moreover, most of the constants appearing 
in our computations through this paper will 
depend on this choice, but we will not make this dependence explicit in our notation.

Some examples of families of probability measures satisfying 
the above assumptions are given in the following.
\begin{example}[Gaussian ensembles]\label{eg:1.3.2}
As in \cite{SZZ}, a natural choice for $\Upsilon_{p}$ 
is taking the sequence 
$\{\eta^{p}_{j}\}_{p\in\mathbb{N},1\leq j\leq d_{p}}$ 
to be i.i.d.\ centered complex Gaussian 
random variables with positive variance. Then the conditions 
\eqref{eq:5.1.2}--\eqref{eq:5.1.3} are clearly satisfied
and in this case $\Upsilon_{p}$ is independent 
of the choice of basis
$O_{p}$.
\end{example}

\begin{example}[Random sections with bounded coefficients]\label{eg:1.3.3}
Let $r_{p}$, $p\in\bN,$ be a sequence of positive numbers 
uniformly bounded below by $r>0$. Let $U_{p}$ denote a complex 
random variable which is uniformly distributed on the disk 
$D(0,r_{p})\subset \C$. For each $p$, we take $\eta^{p}_{j}$, 
$1\leq j\leq d_{p},$ to be a sequence of i.i.d.\ random variables 
with the same distribution as $U_{p}$. Then
\begin{equation}
	\mathbb{E}[|U_{p}|^{d_{p}}]=\frac{2}{d_{p}+2}r_{p}^{d_{p}},
	\label{eq:5.1.10DLM}
\end{equation}
and hence in order to satisfy Condition \eqref{eq:5.1.3} we shall choose $r_{p}\leq 
d_{p}$ for all $p$.
\end{example}

\subsection{Main results for punctured Riemann 
surface}\label{subsection1.4main}
Inspired by the work \cite{SZZ}, we start with investigating 
the supremum norm of random holomorphic sections on open sets.
For this purpose, for $U$ a non-empty open subset of $\Sigma$ and 
$s_{p}\in H^{0}_{(2)}(\Sigma, L^{p})$, we set
\begin{equation}
	\mathcal{M}^{U}_{p}(s_{p})=\sup_{x\in 
	U}|s_{p}(x)|_{h^{p}}<+\infty.
	\label{eq:5.2.1}
\end{equation}
For sections $s_{p}$ of unit $\mathcal{L}^{2}$-norm 
an optimal upper bound for $\mathcal{M}^{U}_{p}(s_{p})$ is 
given by the square root of the supremum of Bergman kernel function 
$B_{p}(x,x)$ on $U$. Using the results of \cite{AMM:20} mentioned above, 
one can get an upper bound for $\mathcal{M}^{U}_{p}(s_{p})$ which 
grows as $p^{1/2}$ if $U$ is relatively compact in $\Sigma$, and as 
$p^{3/4}$ otherwise. Note that for the case of cusp forms on 
arithmetic surfaces (see Section \ref{section4}), 
$\mathcal{M}^{U}_{p}(s_{p})$ has its own interest and such upper 
bounds are also obtained by other methods; we refer to 
\cite{FJK2016,Rudnick2005} for more 
details. 

Our first main result concerns the expectation and concentration properties for the random 
variables $\mathcal{M}^{U}_{p}(s_{p})$. 

\begin{theorem}\label{thm:1.3.2}
Let $\Sigma$ and $(L,h)$ be a punctured Riemann surface
and a line bundle satisfying conditions ($\alpha$) and ($\beta$)
and $\Upsilon_p$ the measures considered in Section \ref{section1.3prob}.
	Let $U$ be an open subset of $\Sigma$ with $\partial U$ 
	having zero measures in $\Sigma$. Then there exists a constant  
	$C_{U}>0$ such that for all $p\in\bN$ we have
	\begin{equation}
		\frac{1}{C_{U}}\,p^{-2}\leq 
		\mathbb{E}[\mathcal{M}^{U}_{p}(s_{p})]\leq C_{U}\,p^{9/4} \,. 
		\label{eq:1.4.2paris}
	\end{equation}
	For any $\delta>0$, there exists a constant $C_{U,\delta}>0$ 
	such that for any $p\in\bN$ we have
	\begin{equation}
		\Upsilon_{p}(\{s_{p}\;:\; |\log{\mathcal{M}^{U}_{p}(s_{p})}|\geq 
		\delta p\})\leq e^{-C_{U,\delta}p^{2}}\,.
		\label{eq:5.2.2}
	\end{equation}
\end{theorem}

For a holomorphic line bundle $E\to\Sigma$ 
and a holomorphic
section $s\in H^0(\Sigma,E)$ 
which is not identically zero 
we denote by $\Div(s)=\sum_{s(x)=0}m_x\cdot x$
the divisor of zeros of $s$, where the sum runs over the zeros $x\in\Sigma$
of $s$ and $m_x=\ord_x(s)$ is the multiplicity of $s$ at $x$.
Note that the zero set of $s$ is closed and discrete, due to the identity
theorem for holomorphic functions.

If $s\in H^{0}(\Sigma,E),$ we define the measure of zeros of $s$ by
\begin{equation}
	[\Div(s)]=\sum_{x\in\Sigma,
	s(x)=0}m_x\,\delta_x\,.
\end{equation}
In view of the higher dimensional case we note that
$[\Div(s)]$ can be identified with a $(1,1)$-current on $\Sigma$. If 
$(\mu_{p})_{p\in\bN}$ is a sequence of $(1,1)$-currents (or measures) 
on $\Sigma$, we say that 
it converges weakly to a $(1,1)$-current $\mu$ on $\Sigma$, if

\begin{equation}\label{eq:1.4.4DLMbis}
\lim_{p\rightarrow\infty}(\mu_{p},\varphi)
=(\mu,\varphi )\quad\text{for all $\varphi\in 
\mathcal{C}_{0}^{\infty}(\Sigma)$}\,, 
\end{equation} 
 where $\mathcal{C}_{0}^{\infty}(\Sigma)$ denotes the space of smooth
 compactly supported functions on $\Sigma$.

Now we go back to our setting where $E=L^{p}, p=1, 2, \cdots$, and 
$s_{p}\in H^{0}_{(2)}(\Sigma, L^{p})$.
If $U\subset \Sigma$ is an open set, we write 
\begin{equation}
	\mathcal{N}^{U}_{p}(s_{p})=\int_U[\Div(s_{p})]
\end{equation}
to denote the 
number of zeros (with multiplicities) of $s_{p}$ in $U$, and $\mathrm{Area}^{L}(U)$ to denote the area of $U$ defined by the 
measures $iR^{L}$. As a consequence of Assumption \ref{item:beta} we have that 
$\mathrm{Area}^{L}(U)$ is finite.

Next we will apply the results in Theorem 
\ref{thm:1.3.2}, using essentially the well-known Poincar\'{e}-Lelong 
formula  (cf.\ \eqref{eq:1.3.10} below), to study the zeros of random 
holomorphic section $s_{p}$. In particular, 
we can infer an upper bound for the hole probabilities. 
Using Borel-Cantelli type arguments we then also obtain the almost sure convergence of zeros of sequences
of holomorphic sections.
For this purpose let us introduce 
the product probability space
 \begin{equation}
 	(\mathcal{H},\Upsilon)=\prod_{p=1}^{\infty} (H^{0}_{(2)}(\Sigma, 
 	L^{p}),\Upsilon_{p}).
\end{equation}
 An element in $\mathcal{H}$ is a sequence $(s_{p})_{p\in\bN}$, 
 $s_{p}\in H^{0}_{(2)}(\Sigma, L^{p})$.

The results we obtain are stated in the following.
\begin{theorem}\label{thm:1.2.1}
Let $\Sigma$ and $(L,h)$ be a punctured Riemann surface
and a line bundle satisfying conditions ($\alpha$) and ($\beta$)
and $\Upsilon_p$ the measures considered in Section \ref{section1.3prob}.

(a) 
$\Upsilon$-almost surely, we have 
the weak convergence of measures
\begin{equation}\label{eq:1.4.4DLM}
\lim_{p\rightarrow\infty}\frac{1}{p}[\mathrm{Div}(s_{p})]
=c_1(L,h)\:\:\text{on $\Sigma$}\,. 
\end{equation} 	
	
(b) If $U$ is an open set of $\Sigma$ with $\partial U$ having zero
measure in $\Sigma$, then for any $\delta>0$, there exists a 
constant $C_{\delta,U}>0$ such that for $p\gg 0$
the following holds:
\begin{equation}\label{eq:1.4.5DLM}
\Upsilon_{p}\Big(\Big\{s_{p}\;:\; 
\Big|\frac{1}{p}\mathcal{N}^{U}_{p}(s_{p})-
\frac{\mathrm{Area}^{L}(U)}{2\pi}\Big|
>\delta\Big\}\Big)\leq e^{-C_{\delta,U}p^{2}}.
\end{equation}
\end{theorem}
We will give a quick proof of item (a) by using Theorem 
\ref{thm:1.3.5}. It follows actually from
\cite[Theorem 5.1]{CM11} that the convergence of currents 
in (a) takes place on $\overline{\Sigma}$.
Our emphasis here is on item (b).
As a consequence of \eqref{eq:1.4.5DLM}, choosing 
$\delta=\mathrm{Area}^{L}(U)/2\pi$ we 
infer the following estimates on the hole probabilities.
\begin{corollary}\label{C:1.2.2}
	If $U$ is a nonempty open set of $\Sigma$ with $\partial U$ having zero
	measure in $\Sigma$, then there exists a 
	constant $C_{U}>0$ such that for $p\gg 0,$
	\begin{equation}
		\Upsilon_{p}(\{s_{p}\;:\; \mathcal{N}^{U}_{p}(s_{p})=0\})
		\leq e^{-C_{U}p^{2}}.
	\end{equation}
\end{corollary}
Note that in the above statements, we can take $U$ to be noncompact in 
$\Sigma$, i.e., an open neighborhood of the punctured points. In particular, for  the 
cusped hyperbolic surfaces investigated in Section 
\ref{section4}, our results can be used to study the zeros of cusp 
forms near cusps.

Moreover, in the case of Gaussian ensembles, 
we also have a lower bound estimate for the 
hole 
probabilities of matching exponential order for Corollary \ref{C:1.2.2}. 
\begin{proposition}\label{prop:1.2.3}	
	Suppose that $\{\Upsilon_{p}\}_{p\in\bN}$ is defined as in 
	Example \ref{eg:1.3.2} with $\sigma_{p}=1$. If $U$ is a relatively compact open subset of $\Sigma$ such that 
	$\partial U$ has zero measure in $\Sigma$, and if there exists a 
	section $\tau\in H^{0}_{(2)}(\Sigma, L)$ such that it does not 
	vanish in $\overline{U}\subset \Sigma$, then there 
	exists $C'_{U}>0$ such that for $p\gg 0$,
	\begin{equation}
		\Upsilon_{p}(\{s_{p}\;:\; \mathcal{N}^{U}_{p}(s_{p})=0\})\geq 
		e^{-C'_{U}p^{2}}.
		\label{eq:1.2.4}
	\end{equation}
	
	Fix an integer $k_{0}\geq 2$. For each $a_{j}\in D$, there exists 
	$r_{j}\in\; ]0,\frac{1}{2}[$ 
	and $\tau_{j}\in H^{0}_{(2)}(\Sigma, L^{k_{0}})$
	such that $\tau_{j}$ has no zeros in 
	$\mathbb{D}^{*}_{2r_{j}}\subset V_{j}$ described in Assumption 
	\ref{item:beta}. For $0<r<r_{j}$, set $\mathbb{D}(r,r_{j})=\{z\in 
	\C\;:\; r< |z|< r_{j}\}\subset\mathbb{D}^{*}_{r_{j}}\subset 
	V_{j}$. Then there exists $c_{j}>0$ such that for $0<r<r_{j}$, we have
	\begin{equation}
		\Upsilon_{pk_{0}}\big(\big\{s_{pk_{0}}\;:\; 
		\mathcal{N}^{\mathbb{D}(r,r_{j})}_{pk_{0}}(s_{pk_{0}})=0\big\}\big)\geq 
		e^{- c_{j}|\log{r}|p^{2}}=r^{c_{j}p^{2}},\;\forall\, p\gg 0.
		\label{eq:1.2.5}
	\end{equation}
\end{proposition}

In the next 
subsection we provide some intermediate results, which are of independent interest and which will play an important role on our way to proving the results given above.

\subsection{Intermediate results: an approach to Theorem 
\ref{thm:1.2.1}}\label{subs1.3}
The normalized Bergman kernel is defined as
\begin{equation}
	P_{p}(x,y)=\frac{|B_{p}(x,y)|_{h^{p}_{x}\otimes 
	h_{y}^{p,\ast}}}{\sqrt{B_{p}(x,x)}\sqrt{B_{p}(y,y)}}, \quad  x,y\in \Sigma.
	\label{eq:1.3.1}
\end{equation}
A near-diagonal estimate on $P_{p}(x,y)$ plays a central role in our 
computations. In the case of compact K\"{a}hler 
manifolds, such results were established in \cite[Propositions 2.6 and 2.7]{SZ08} as well as in \cite[Proposition 2.1]{SZZ}. In our setting, we 
will take advantage of the Bergman kernel expansion for 
complete, possibly noncompact, Hermitian manifolds obtained by Ma and 
Marinescu in \cite[Theorems 4.2.1 \& 6.1.1]{MM07}.

\begin{theorem}\label{thm:1.3.1}
	Let $U$ be a relatively compact open subset of $\Sigma$, then we have 
	the following uniform estimate on the normalized Bergman kernel.
	Fix $k\geq 1$ and $b>\sqrt{16k/\varepsilon_{0}}$. Then: 
	
\begin{enumerate} [label=(\alph*)]
\item 	There exists $C>0$ such that for all $p\in\bN_{\geq 2}$ 
and for all $x,y\in U$ 
with $\mathrm{dist}(x,y) \geq b\sqrt{\frac{\log p}{p}}$ we have
$P_{p}(x,y)\leq C p^{-k}$.

\item
For 
$p\geq 2$, there exist functions 
$G_p:\Big\{(x,y)\in U\times U:\mathrm{dist}(x,y) \leq b\sqrt{\frac{\log p}{p}}
\Big\}\to\R$
such that  $\sup G_p\to 0$ as $p\rightarrow\infty$ such that
	\begin{equation}
		P_{p}(x,y)= 
		 (1+ G_{p}(x,y))\exp\Big(-\frac{a(x)p}{4}\mathrm{dist}(x,y)^{2} \Big).
		\label{eq:1.5.2July}
	\end{equation}
\end{enumerate}
\end{theorem}

Note that under the higher dimensional setting in Subsection \ref{section1.6}, an 
analog of the above results still holds true (cf.\ Theorem 
\ref{thm:5.1.1}). These estimates, together with a crucial inequality for the marginal 
densities of $\Upsilon_{p}$ proved in Proposition 
\ref{prop:boundmarginal}, are the key ingredients of our proof of 
\eqref{eq:5.2.2} in Theorem \ref{thm:1.3.2}. 
As a  consequence, we obtain the following proposition.

\begin{proposition}\label{prop:1.3.3}
	Let $U$ be a relatively compact open subset in $\Sigma$ such that 
	$\partial U$ has zero measure in $\Sigma$. For any $\delta>0$, 
	there exists $C_{U,\delta}>0$ such that for all $ p\gg 0$,
	\begin{equation}
		\Upsilon_{p}\Big(\Big\{s_{p}\;:\; 
		\int_{U}\big|\log{|s_{p}|_{h^{p}}}\big| \,\omega_{\Sigma}\geq 
		\delta p\Big \} \Big)\leq e^{-C_{U,\delta}p^{2}}.
		\label{eq:1.3.5}
	\end{equation}
\end{proposition}

Note that the estimate \eqref{eq:1.3.5} is a version of 
\cite[Lemma 1.6]{SZZ}. To prove it, we use here Theorem 
\ref{thm:1.3.2} (cf.\ Subsection \ref{subs3.3}) instead of 
\cite[Theorem 3.1]{SZZ}. But since $\omega_{\Sigma}$ 
is singular near punctures, the 
estimate \eqref{eq:1.3.5} does not hold if we take $U=\Sigma$. 
Indeed, as 
we will see in Subsection \ref{subs3.3}, 
$\big|\log{|s_{p}|_{h^{p}}}\big|$ is not integrable with respect to 
$\omega_{\Sigma}$ near the punctures.

Using the Poincar\'{e}-Lelong formula, Proposition 
\ref{prop:1.3.3} leads us to the next result, so 
that Theorem \ref{thm:1.2.1} will be one of its consequences.

\begin{theorem}\label{thm:1.3.5}
	If $\varphi\in \mathcal{C}^{\infty}(\overline{\Sigma})$ is such that 
	$\varphi$ is locally constant in an open neighborhood of 
	$D$, then for $\delta>0$, there 
	exists $C_{\varphi,\delta}>0$ such that for $p\gg 0$, we have
	\begin{equation}
		\Upsilon_{p}\Big( \Big\{s_{p}\;:\; 
		\Big| \Big( \frac{1}{p}[\mathrm{Div}(s_{p})],\varphi \Big)
		-\int_{\Sigma}\varphi c_{1}(L,h)\Big|>\delta \Big \} \Big)\leq 
		e^{-C_{\varphi,\delta}p^{2}},
		\label{eq:3.4.11}
	\end{equation}
	where the sum in \eqref{eq:3.4.11} is taking into account the 
	multiplicities of the zeros.
\end{theorem}

We would like to point out the difference between \eqref{eq:3.4.11} 
here and the one proved in \cite[Theorem 1.5]{SZZ}. Indeed, for $p\geq 2$, 
the section $s_{p}$ always vanishes at the punctures as specified by $D$. Denoting by
$\mathrm{ord}_{a_{j}}(s_{p})\geq 1$ the vanishing order of $s_{p}$ at 
$a_{j} \in D$, we infer that
\begin{equation}
	\mathcal{N}^{\overline{V}_{j}}_{p}(s_{p})= 
	\mathcal{N}^{V_{j}}_{p}(s_{p})+\mathrm{ord}_{a_{j}}(s_{p}),
	\label{eq:1.3.8}
\end{equation}
where $\overline{V}_{j}$, $V_{j}$ are open sets as in Assumptions 
\ref{item:alpha} and \ref{item:beta}. In terms of divisors on $\overline{\Sigma}$, we can then rewrite 
\eqref{eq:1.3.8} as 
\begin{equation}
	[\mathrm{Div}_{\overline{\Sigma}}(s_{p})]= 
	[\mathrm{Div}(s_{p})]+\sum_{j}\mathrm{ord}_{a_{j}}(s_{p})\delta_{a_{j}},
	\label{eq:1.3.9}
\end{equation}
where we view $[\mathrm{Div}(s_{p})]$ as a divisor on 
$\overline{\Sigma}$.

Note that $h$ is a singular Hermitian metric of $L$ over 
$\overline{\Sigma}$, but for any smooth function $\varphi$ on 
$\overline{\Sigma}$, the Poincar\'{e}-Lelong formula still holds true \cite[Theorem 2.3.3]{MM07}, i.e.,
\begin{equation}
	( 
	[\mathrm{Div}_{\overline{\Sigma}}(s_{p})],\varphi 
	)=\frac{i}{\pi}(\partial\overline{\partial}\log|s_{p}|_{h^{p}},\varphi)+p ( c_{1}(L,h),\varphi ).
	\label{eq:1.3.10}
\end{equation}

Comparing \eqref{eq:1.3.9} and \eqref{eq:1.3.10} with the event in \eqref{eq:3.4.11}, we see that in order to obtain Theorem 
\ref{thm:1.3.5}, it is sufficient to control the vanishing orders $\mathrm{ord}_{a_{j}}(s_{p})$ in a 
uniform way for $p\gg 0$ and for arbitrary $s_{p},$ 
except for possibly subsets of small probability. Indeed, we have the following result.
\begin{lemma}\label{lm:1.3.6}
	There exist $p_{0}> 0$, $k_{0}>0$ such that for any $p\geq 
	p_{0}$, the following inequalities hold $\Upsilon_{p}$-almost 
	surely,
	\begin{equation}
		\mathrm{ord}_{a_{j}}(s_{p})\leq k_{0}, \;\forall\, a_{j}\in D.
		\label{eq:1.3.11}
	\end{equation}
\end{lemma}
This lemma will be restated as Lemma \ref{lm:3.4.1} in a more 
concrete way, and its proof relies on the positivity of $L$ on 
$\overline{\Sigma}$, which is given in Subsection \ref{subs3.4}.

\subsection{Higher dimensional Hermitian 
manifolds}\label{section1.6}

In Section \ref{section:higherdim}, we provide extensions of 
our results (with suitable adaptations) to higher dimensional complex manifolds. Since our 
method relies on the Bergman kernel expansions, we adopt the 
geometric settings as in \cite[Chapter 6]{MM07} 
and \cite{DMS12}. 

Let $(X,J,\omega)$ be an  
$m$-dimensional complex Hermitian (not necessarily compact) manifold 
where $J$ denotes the complex structure and $\omega$ is a positive 
$(1,1)$ form. To $\omega$ we associate a $J$-invariant 
Riemannian metric $g^{TX}$ defined by $g^{TX}(u,u)=\omega(u,Jv)$ for 
all $u,v\in T_xX$ and $x\in X$. We assume that $(X,g^{TX})$ is 
complete. If $U\subset X$ is open, let $\Omega^{p,q}_{0}(U)$ denote 
the set of smooth differential forms on $U$ of bi-degree $(p,q)$ which have 
compact support in $U$. In particular, 
$\mathcal{C}^{\infty}_{0}(U)=\Omega^{0,0}_{0}(U)$.

Let $(L,h)$ be a holomorphic line bundle over $X$. We still denote the 
Chern curvature form of $L$ by $R^L$, and let $R^{\mathrm{det}}$ be the 
curvature of the holomorphic connection $\nabla^{\mathrm{det}}$ on 
$K_X^*=\det(T^{(1,0)}X)$ with the Hermitian metric induced by 
$g^{TX}$. In addition we assume that there exists 
$\varepsilon_{1}>0$, $C_{1}>0$ such that
\begin{equation}\label{A1}
	i R^{L}>\varepsilon_{1} \omega,\ \ \ \ \ i R^{\det}>-C_{1} 
	\omega, \ \ \ \ \  \ \ |\partial \omega|_{g^{TX}}<C_{1}.
\end{equation} 
Some remarks:
\begin{enumerate}
	\item If $(X,\omega)$ is K\"ahler then $\partial \omega=0$ and the 
	second condition in (\ref{A1}) is trivially satisfied. Moreover, in 
	this case, $i R^{\det}=\mathrm{Ric}_{\omega}$, where 
	$\mathrm{Ric}_{\omega}$ is the Ricci curvature associated with 
	$g^{TX}$.
	\item The assumptions in (\ref{A1}) are the necessary
	conditions to ensure that one can apply the 
	H\"{o}rmander-Andreotti-Vesentini $L^{2}$-estimates for $\overline{\partial}.$
	In our context, these conditions imply actually the asymptotics
	of the Bergman kernel on compact sets of $X$ (cf. \cite[Theorem 6.1.1]{MM07}).
\end{enumerate}

Let $\mathcal{C}_0^{\infty}(X,L^p)$ denote the space of 
compactly supported smooth sections on which we define a scalar inner product by
\begin{equation}\label{norm}
	\langle s_1,s_2\rangle:=\int_X\langle s_1(x),s_2(x)\rangle_{h_p} 
	\mathrm{dV}(x)
\end{equation}
where $h^{p}=(h^{L})^{\otimes p}$ and $\mathrm{dV}=\frac{1}{m!}\omega^m$ is 
the volume form induced by $\omega$. We also let $\mathcal{L}^{2}(X,L^p)$ 
be the Hilbert space obtained by completing  $\mathcal{C}_0^{\infty}(X,L^p)$
with respect to the norm $\|\cdot\|_p$ induced by (\ref{norm}). 
Here we consider Hilbert space of holomorphic sections
\begin{equation}
	H_{(2)}^0(X,L^p):= \mathcal{L}^{2}(X,L^p) \cap H^0(X,L^p).
	\label{eq:6.0.3}
\end{equation}

In addition, we assume that for $p\in\mathbb{N}$, 
$d_{p}=\dim_{\C} H_{(2)}^0(X,L^p)$ is finite, and that as 
$p\rightarrow \infty$,
\begin{equation}
	d_{p}=\mathcal{O}(p^{m}).
	\label{eq:6.0.4}
\end{equation}
This hypothesis is satisfied in several geometric situations. The 
punctured Riemann surface discussed in previous subsections is an 
example of complex dimension one. We will give other examples in the 
Section \ref{section:higherdim}.

For $s_{p}\in H^{0}_{(2)}(X,L^{p})$ let $Z_{s_{p}}$ denote the zero 
set of $s_{p}$, i.e.,
\begin{equation}
	Z_{s_{p}}=\{x\in X\;:\; s_{p}(x)=0\}.
	\label{eq:6.1.8}
\end{equation}
For a nonzero $s_{p}$, $Z_{s_{p}}$ is a complex $(m-1)$-dimensional 
hypersurface. 
We define the divisor of $s_p$ by $\Div (s_p)=
\sum_V \ord_V (s_p)\cdot V$
where the sum runs over all irreducible analytic hypersurfaces $V$ of $Z_{s_p}$ and
$\ord_V (s_p)\in\Z$ is the order of $s_p$ along $V$.
For any hypersurface $V$ we denote by $[V]$ the current
of integration on $V$ and by 
$[\Div(s_{p})]=\sum_V \ord_V (s_p)[V]$ the current of
integration on $\Div(s_{p})$.
This is a $(1,1)$-current.

Consider
the product probability space
 \begin{equation}
 	(\mathcal{H},\Upsilon)=\prod_{p=1}^{\infty} (H^{0}_{(2)}(X, 
 	L^{p}),\Upsilon_{p}).
\end{equation}
When $\Upsilon_{p}, p\in\bN$, are defined from Gaussian ensembles 
(Example \ref{eg:1.3.2}), Dinh, Marinescu and Schmidt \cite[Theorem 1.2]{DMS12} showed that the zero-divisors of 
generic random sequences $(s_{p})_{p}\in \prod^{\infty}_{p=1} 
H_{(2)}^0(X,L^p)$ are equidistributed with respect to 
$c_{1}(L,h^{L})$. For proving this result, they actually gave a 
convergence speed for the divisors as follows. 
\begin{theorem}[\protect{\cite[Theorem 1.5]{DMS12}}]
	\label{thm:6.1}
	If $U$ is a relatively compact open subset of $X$, then there 
	exists a constant $c=c(U)>0$ and a positive integer $p(U)$ with 
	the following property. For any positive number sequence 
	$(\lambda_{p})_{p\in\bN}$ 
	with $\lim_{p\rightarrow \infty}\lambda_{p}/\log{p}=\infty$, and 
	for any $p\geq p(U)$ and $\varphi\in \Omega^{m-1,m-1}_{0}(U)$, we have
	\begin{equation}
		\Upsilon_{p}\Big(\Big\{s_{p}\;:\; 
		\bigg|\Big(\frac{1}{p}[\mathrm{Div}(s_{p})] 
		-c_{1}(L,h),\varphi\Big)\bigg|>\frac{\lambda_{p}}{p}\Vert \varphi\Vert _{\mathcal{C}^{2}} \Big\}\Big)
		\leq 
		cp^{2m}e^{-\lambda_{p}/c},
		\label{eq:6.0.5}
	\end{equation}
	where $\Vert \cdot\Vert _{\mathcal{C}^{2}}$ denote the 
	$\mathcal{C}^{2}$-norm of smooth sections.
\end{theorem}

Following \cite[Theorem 1.1]{SZZ}, if we want to get the probability 
bound like $e^{-c'p^{m+1}}$ 
in \eqref{eq:6.0.5}, we shall take the sequence $\lambda_{p}=p^{m+1}$, 
thus \eqref{eq:6.0.5} gives
\begin{equation}
	\Upsilon_{p}\Big(\Big\{s_{p}\;:\; 
	\bigg|\Big(\frac{1}{p}[\mathrm{Div}(s_{p})] - 
	c_{1}(L,h),\varphi\Big)\bigg|> p^{m}\Vert \varphi\Vert _{\mathcal{C}^{2}} \Big\}\Big)\leq 
	cp^{2m}e^{-p^{m+1}/c}.
	\label{eq:6.0.6}
\end{equation}
This is clearly a weaker version of the estimate as in Theorem 
\ref{thm:1.2.1}. 

Now let $\Upsilon_{p}, p\in\bN$ be the probability measures on 
$H^{0}_{(2)}(X,L^{p})$, $p\in\bN$ 
constructed in Subsection \ref{section1.3prob} (not necessarily assumed to be 
Gaussian). Note that in Condition \eqref{eq:5.1.3}, we have 
$d_{p}=\mathcal{O}(p^{m})$.
For this higher dimensional setting, we will 
prove the following results. 

\begin{theorem}\label{thm:1.6.1}
	Let $U$ be a relatively compact open subset of $X$ with $\partial U$ 
	having zero measures in $X$. Then there exists a constant  
	$C_{U}>0$ such that for any $p\in\bN$
	\begin{equation}
		\frac{1}{C_{U}}\,p^{-m-1}\leq 
		\mathbb{E}[\mathcal{M}^{U}_{p}(s_{p})]\leq C_{U}\,p^{2m}.
		\label{eq:1.6.7paris}
	\end{equation}
	For any $\delta>0$, there exists a constant $C_{U,\delta}>0$ 
	such that for any $p\in\bN$, 
	\begin{equation}
		\Upsilon_{p}(\{s_{p}\;:\; |\log{\mathcal{M}^{U}_{p}(s_{p})}|\geq 
		\delta p\})\leq e^{-C_{U,\delta}p^{m+1}}.
		\label{eq:1.6.8paris}
	\end{equation} 
\end{theorem}

Then we can get the following improvement of \eqref{eq:6.0.6}.

\begin{theorem}
	\label{thm:6.2}
	If $U$ is a relatively compact open subset of $X$, then for any 
	$\delta >0$ and $\varphi\in \Omega^{m-1,m-1}_{0}(U)$, there 
	exists a constant $c=c(U,\delta,\varphi)>0$ such that for $p\in\bN$,
	we have
	\begin{equation}
		\Upsilon_{p}\left(\left\{s_{p}\;:\; 
		\bigg|\Big(\frac{1}{p}[\mathrm{Div}(s_{p})] 
		-  c_{1}(L,h),\varphi\Big)\bigg|>\delta \right\}\right)\leq 
		e^{-c\,p^{m+1}}.
		\label{eq:6.0.7}
	\end{equation}
	Moreover, $\Upsilon$-almost surely we have the weak convergence of $(1,1)$-currents,
 \begin{equation}\label{eq:1.6.2DLM}
\lim_{p\rightarrow\infty}\frac{1}{p}[\mathrm{Div}(s_{p})]
=c_1(L,h)\,.
	 	\end{equation} 	

\end{theorem}

Since $c_{1}(L,h)$ is positive, then $\frac{c_{1}(L,h)^{m}}{m!}$ 
defines a positive volume element on $X$. If $U\subset X$ is open, set
\begin{equation}
	\mathrm{Vol}^{L}_{2m}(U)=\int_{U} \frac{c_{1}(L,h)^{m}}{m!}.
	\label{eq:1.6.8DLM}
\end{equation}
We will see in \eqref{eq:5.1.2DLM} that this volume is always finite.

For $s_{p}\in H^{0}_{(2)}(X,L^{p})$, we define the $(2m-2)$-dimensional volume (with 
respect to $c_{1}(L,h)$) of $Z_{s_{p}}$ in an 
open subset $U\subset X$ as follows,
\begin{equation}
	\mathrm{Vol}^{L}_{2m-2}(Z_{s_{p}}\cap U)=\int_{Z_{s_{p}}\cap U} 
	\frac{c_{1}(L,h)^{m-1}}{(m-1)!}\,\cdot
	\label{eq:6.1.9}
\end{equation}
As a consequence of Theorem \ref{thm:6.2}, we have the following 
theorem.
\begin{theorem}
	\label{thm:6.3}
	If $U$ is a relatively compact open subset of $X$ such that 
	$\partial U$ has zero measure in $X$, then for any 
	$\delta >0$, there 
	exists a constant $c_{U,\delta}>0$ such that for $p$ 
	large enough,
	we have
	\begin{equation}
		\Upsilon_{p}\Big(\Big\{s_{p}\;:\; 
		\Big|\frac{1}{p}\mathrm{Vol}^{L}_{2m-2}(Z_{s_{p}}\cap U) 
		- m\mathrm{Vol}^{L}_{2m}(U) \Big|>\delta \Big \}\Big)\leq 
		e^{-c_{U,\delta} p^{m+1}}.
		\label{eq:6.1.10}
	\end{equation}
	
	If $U$ is a nonempty open (possibly not relatively compact) set of $X$ with $\partial U$ having zero
	measure in $X$, then there exists a 
	constant $C_{U}>0$ such that
	\begin{equation}
		\Upsilon_{p}(\{s_{p}\;:\; Z_{s_{p}}\cap U=\emptyset\})
		\leq e^{-C_{U}p^{m+1}},\, \forall\, p\gg 0.
		\label{eq:1.6.14paris}
	\end{equation}
\end{theorem}
Note that when $X$ is compact and $\omega=iR^{L}$, as well as for the special choice of $\Upsilon_{p}, 
p\in\bN,$ as a Gaussian ensemble
(cf.\ Example \ref{eg:1.3.2}), the results in Theorems \ref{thm:1.6.1}, 
\ref{thm:6.2} and \ref{thm:6.3} are exactly the main results proved in 
\cite{SZZ}.

\subsection{Organization of the paper}
This paper is organized as follows. In Section \ref{section2bergman}, 
we recall the estimates of the Bergman kernels for the punctured 
Riemann surface $\Sigma$. In Subsection \ref{section2.3DLM}, we 
give a proof of Theorem \ref{thm:1.3.1}.

In Section \ref{section3DLM}, we give the proofs of other
results stated in Subsections \ref{subsection1.4main} \& 
\ref{subs1.3}. In particular, Subsections \ref{section3.1DLM} -- 
\ref{section3.3DLM} are devoted to prove Theorem \ref{thm:1.3.2}. 
In Subsection \ref{subs3.3}, we only sketch a Proof of Proposition 
\ref{prop:1.3.3}, since most of the arguments follow from 
\cite[Subsection 4.1]{SZZ}. 
In Subsection \ref{subs3.4}, we prove at 
first Lemma \ref{lm:1.3.6}, and then prove Theorem \ref{thm:1.3.5}. 
In Subsection \ref{subs3.5}, we prove Theorem \ref{thm:1.2.1} using 
Theorem 
\ref{thm:1.3.5}. At last, in Subsection \ref{subs3.6}, we prove 
Proposition \ref{prop:1.2.3}.

In Section \ref{section4}, we give a discussion for hyperbolic 
surfaces with cusps and of high genus, they are important examples of punctured 
Riemann surfaces where our results apply.

Finally, in Subsection \ref{section:higherdim}, we study the higher 
dimensional complex Hermitian manifolds, and give the proofs of the 
results stated in Subsection \ref{section1.6}.

\bigskip 

{\bf Acknowledgment:} We gratefully acknowledge support of 
DFG Priority Program 2265 \lq Random Geometric Systems\rq. 
The authors thank Dominik Zielinski for useful discussions.

\section{Estimates on Bergman kernel}\label{section2bergman}
In this section, we recall some results on the Bergman kernel 
expansions for our punctured Riemann surface $\Sigma$ obtained by 
Ma-Marinescu \cite[Chapter 6]{MM07} and by Auvray-Ma-Marinescu 
\cite{AMM:20}. Note that the results in \cite[Chapter 6]{MM07} are 
applicable to general Hermitian manifolds and line bundles such as 
the ones in Subsection \ref{section1.6}; we refer to Section 
\ref{section:higherdim} for a more detailed discussion. In this section, we focus on 
$\Sigma$.
\subsection{On-diagonal estimates}
Recall that the positive smooth function $a$ on $\Sigma$ is defined as 
follows, for $x\in \Sigma$,
\begin{equation}
	a(x)=\frac{iR^{L}_{x}}{\omega_{\Sigma,x}}\geq \varepsilon_{0}.
	\label{eq:2.1.1DLM}
\end{equation}

In our setting (with Assumptions \ref{item:alpha} and \ref{item:beta}), due to \cite[Theorem 6.1.1]{MM07} we have the following result.
\begin{theorem}\label{thm:2.1.1}
	For any compact set $K\subset \Sigma$, we have the uniform 
	asymptotic expansion for $x\in K$,
	\begin{equation}
		B_{p}(x,x)=\frac{p}{2\pi}a(x)+\mathcal{O}_{K}(1),\;\mathrm{as\;}p\rightarrow +\infty.
		\label{eq:2.1.1}
	\end{equation}
\end{theorem}

In \cite{AMM:20}, an asymptotic expansion of $B_{p}$ near 
the punctured points is obtained by studying the Bergman kernel 
expansion for the punctured disk endowed with the Poincar\'{e} metric. 
Furthermore, they obtained a global optimal upper bound for $B_{p}$. By \cite[Corollary 1.4]{AMM:20}, we have
\begin{equation}
	\sup_{x\in \Sigma}B_{p}(x,x)=\Big(\frac{p}{2\pi}\Big)^{3/2}+\mathcal{O}(p), 
	\;\mathrm{as\;}p\rightarrow +\infty.
	\label{eq:2.1.3}
\end{equation}

\begin{remark}
The uniform upper-bound of $B_{p}(x,x)$ given in \eqref{eq:2.1.3} plays 
an important role in the Proof of Theorem \ref{thm:1.3.2}. 
In the absence of such uniform upper-bound on the noncompact 
manifold, we should assume $U$ to be relatively compact 
in Theorem \ref{thm:1.3.2}. 
As we will see in Subsection \ref{subs3.6}, the upper-bound of 
$B_{p}(x,x)$ in \eqref{eq:2.1.3} is also necessary in the proof of 
\eqref{eq:1.2.5}.
\end{remark}

\subsection{Off- and near-diagonal estimates}
For the off- and near-diagonal expansion of $B_{p}$, we still apply 
\cite[Theorem 6.1.1]{MM07} to our punctured Riemann surface.

\begin{proposition}[{\cite[Theorem 6.1.1]{MM07}}]\label{prop:2.2.1}
	For any $\ell\in \mathbb{N}$ and $\delta>0$, for any compact subset 
	$K\subset \Sigma$, there exists $C_{\ell,\delta,K}>0$ such that for 
	all $p\in\bN$ and $x,y\in K$ with $\mathrm{dist}(x,y)\geq 
	\delta$,
	\begin{equation}
		|B_{p}(x,y)|\leq C_{l,\delta,K}\, p^{-\ell}.
		\label{eq:2.2.1paris}
	\end{equation}
	
	Fix any compact subset $K$, and for any $N\in \mathbb{N}$, there 
	exist $\varepsilon>0,$ functions  $\mathcal{F}_{r},$ and constants $C,C'>0$ such that for 
	$x_{0}\in K$, $v,v'\in (T_{x_{0}}\Sigma, g^{T\Sigma}_{x_{0}})$, 
	$\Vert v\Vert,\Vert v'\Vert\leq 2\varepsilon$, we have as 
	$p\to\infty$,
	\begin{equation}
		\begin{split}
			&\bigg|\frac{1}{p}B_{p}(\exp_{x_{0}}(v),\exp_{x_{0}}(v'))-\sum_{r=0}^{N}\mathcal{F}_{r}(\sqrt{p}v,\sqrt{p}v')\kappa^{-1/2}(v)\kappa^{-1/2}(v')p^{-r/2}\bigg|\\
			&\leq 
Cp^{-(N+1)/2}(1+\sqrt{p}\Vert v\Vert+\sqrt{p}\Vert 
v'\Vert)^{2N+6}\exp(-C'\sqrt{p}\Vert v-v'\Vert)+\mathcal{O}(p^{-\infty})\,.
\end{split}
\label{eq:2.2.2}
\end{equation}
\end{proposition}
The norm in left-hand side of \eqref{eq:2.2.2} is taken at the point $x_{0}$ 
after trivializing the line bundle $L$ near $x_{0}$ along the 
radial geodesic path centered at $x_{0}$ with respect to the 
Chern connection of $(L,h)$. The function $\kappa$, 
and the functions $\mathcal{F}_{r}$, 
$r\in\bN$, all depending 
smoothly on $x_{0}$, will be 
described more explicitly below (cf.\ \eqref{eq:2.2.3July},\;\eqref{eq:2.2.3}). 
The term $\mathcal{O}(p^{-\infty})$ is used to denote a decay faster than 
$p^{-\ell}$ for any $\ell\in\bN$.

As in the above proposition, for $x_{0}\in K$ and $\varepsilon>0$ 
sufficiently small, we can identify the Euclidean ball 
$B^{T_{x_{0}}\Sigma}(0,4\varepsilon)\subset (T_{x_{0}}\Sigma, g^{T\Sigma}_{x_{0}})$ 
with the geodesic ball $B^{\Sigma}(x_{0},4\varepsilon)\subset \Sigma$ via the 
local geodesic coordinate centered at $x_{0}$. Let 
$g^{\Sigma_{0}}$ be a metric on $\Sigma_{0}:=T_{x_{0}}\Sigma\simeq \R^2$ which 
coincides  with $g^{T\Sigma}$ on 
$B^{T_{x_{0}}\Sigma}(0,2\varepsilon)$, and $g^{T\Sigma}_{x_{0}}$ 
outside $B^{T_{x_{0}}\Sigma}(0,4\varepsilon)$. Let $dv_{\Sigma_{0}}$ 
be the Riemannian volume form of $(\Sigma_{0}, g^{\Sigma_{0}})$, and 
let $dv_{T_{x_{0}}\Sigma}$ denote the Riemannian volume form of  
$(T_{x_{0}}\Sigma, g^{T\Sigma}_{x_{0}})$. The function $\kappa$ is a 
positive function on $T_{x_{0}}\Sigma$ such that for $v\in 
T_{x_{0}}\Sigma$,
\begin{equation}
	dv_{\Sigma_{0}}(v)=\kappa(v)dv_{T_{x_{0}}\Sigma}(v).
	\label{eq:2.2.3July}
\end{equation}
In particular, $\kappa(0)=1$. Moreover, when $x_{0}$ varies in the 
compact set $K$, for $v\in 
T_{x_{0}}\Sigma$ with $\Vert v\Vert\leq 2\varepsilon$, 
the function $\kappa(v)$ is uniformly bounded.

To describe the function $\mathcal{F}_{r}$, we need to explain the 
complex coordinate near $x_{0}$. Let $\mathbf{f}$ denote a unit 
vector of $T^{(1,0)}_{x_{0}}\Sigma$, i.e. 
$g^{T\Sigma}_{x_{0}}(\mathbf{f},\overline{\mathbf{f}})=1$. Set
\begin{equation}
	\mathbf{e}_{1}=\frac{1}{\sqrt{2}}(\mathbf{f}+\overline{\mathbf{f}}),\;  
	\mathbf{e}_{2}=\frac{i}{\sqrt{2}}(\mathbf{f}-\overline{\mathbf{f}}).
	\label{eq:2.2.3DLM}
\end{equation}
Then $\{\mathbf{e}_{1},\mathbf{e}_{2}\}$ is an oriented orthonormal 
basis of the (real) tangent space $(T_{x_{0}}\Sigma, 
g^{T\Sigma}_{x_{0}})$. If 
$v=v_{1}\mathbf{e}_{1}+v_{2}\mathbf{e}_{2}\in T_{x_{0}}\Sigma $, 
$v_{1}, v_{2}\in\R$, then
\begin{equation}
	v=(v_{1}+i v_{2})\frac{1}{\sqrt{2}}\mathbf{f} +(v_{1}-i 
	v_{2})\frac{1}{\sqrt{2}}\overline{\mathbf{f}},
\end{equation}
and we associate it with a complex coordinate 
$z=v_{1}+i v_{2}\in\C$. In this coordinate, we have 
$\frac{\partial}{\partial z}= \frac{1}{\sqrt{2}}\mathbf{f}$, and 
$\Vert \frac{\partial}{\partial z}\Vert=|\frac{\partial}{\partial z}|_{g^{T\Sigma}}=\frac{1}{2}$. Note that, 
for $z\in\C$, $|z|$ still denotes the standard norm of $z$ as 
complex number.

Now, for $v,v'\in T_{x_{0}}\Sigma$, let $z$, $z'$ denote the 
corresponding complex coordinates. Set
\begin{equation}
	\mathcal{F}_{r}(v,v')=\mathcal{P}(v,v')\mathcal{J}_{r}(v,v'),
	\label{eq:2.2.3}
\end{equation}
where
\begin{equation} \label{eq:Pdef}
	\mathcal{P}(v,v')=\frac{a(x_{0})}{2\pi}\exp \Big (-\frac{1}{4}a(x_{0} )(|z|^{2}+|z'|^{2}-2z\overline{z}') \Big),
\end{equation}
and 
\begin{align} \label{eq:Jr}
\begin{split}
\mathcal{J}_{r}(v,v') &\text{ is a polynomial in $v,v'$ of degree at most $ 3r$,} \\
&\text{whose 
coefficients are smooth in $x_{0}\in \Sigma$.}
\end{split}
\end{align} 
In particular,
\begin{equation} \label{eq:J0}
	\mathcal{J}_{0}=1.
\end{equation}

The following lemma is elementary.
\begin{lemma}\label{lm:2.2.2}
	The norm of $\mathcal{P}$ satisfies
	\begin{equation}
		|\mathcal{P}(v,v')|=\frac{a(x_{0})}{2\pi}\exp \Big(-\frac{1}{4}a(x_{0})\Vert v-v'\Vert ^{2} \Big).
		\label{eq:2.2.5}
	\end{equation}
\end{lemma}
\begin{proof}
	This follows directly from \eqref{eq:Pdef} in combination with the formula
	\begin{equation}
		\begin{split}
			|z-z'|^{2}&=|z|^{2}+|z'|^{2}-2\Re(z\overline{z}')\\
			&=|z|^{2}+|z'|^{2}-2z\overline{z}' +2i\Im(z\overline{z}'),
		\end{split}
		\label{eq:2.2.6}
	\end{equation}
	where $\Re(\cdot)$, $\Im(\cdot)$ denote, respectively, the real 
	and imaginary parts.
	By definition, we have $|z-z'|=\Vert v-v'\Vert $.
\end{proof}

\subsection{Proof of Theorem \ref{thm:1.3.1}}\label{section2.3DLM}
We divide our proof into two steps as follows. 

\noindent
\textbf{Step 1:} We start with proving the first estimate in the 
theorem.
Note that $U$ is relatively compact in $\Sigma$, so 
$\overline{U}$ is compact and Proposition \ref{prop:2.2.1} is 
applicable. Let $\varepsilon>0$ be the sufficiently small quantity 
stated in the second part of Proposition \ref{prop:2.2.1}. Then by 
the first part of the same proposition, if 
$x,y\in U$ is such that $\mathrm{dist}(x,y)\geq \varepsilon$, we have
\begin{equation}
	|B_{p}(x,y)|\leq C_{k+1,\varepsilon,K}\, p^{-k-1}.
	\label{eq:3.1.1}
\end{equation}
We fix a large enough $p_{0}\in\mathbb{N}$ such that
\begin{equation}
	b\sqrt{\frac{\log{p_{0}}}{p_{0}}}\leq \frac{\varepsilon}{2}.
	\label{eq:3.1.2}
\end{equation}
For $p>p_{0}$, if $x,y\in U$ is such that 
$b\sqrt{\frac{\log{p}}{p}}\leq\mathrm{dist}(x,y)< \varepsilon$, then 
we take advantage of the expansion in \eqref{eq:2.2.2} with $N=2k+1$, $x_{0}=x$, 
$v=0$, 
$y=\exp_{x}(v')$, and $v'\in T_{x}\Sigma$, in order to obtain
\begin{equation}
	\begin{split}
		&\bigg|\frac{1}{p}B_{p}(x,y)-\sum_{r=0}^{2k+1}\mathcal{F}_{r}(0,\sqrt{p}v')\kappa^{-1/2}(v')p^{-r/2}\bigg|\\
		&\leq 
		Cp^{-k-1}(1+\sqrt{p}\Vert v'\Vert )^{4k+8}\exp(-C'\sqrt{p}\Vert v'\Vert )+\mathcal{O}(p^{-k-1}).
	\end{split}
	\label{eq:3.1.3}
\end{equation}
Now for $k \ge 1$, there exists a constant $C_{k}>0$ such that for any $r>0$,
\begin{equation}
	(1+r)^{4k+8}\exp(-C'r)\leq C_{k}.
\end{equation}
Note that $\Vert v'\Vert =\mathrm{dist}(x,y)$. By \eqref{eq:2.2.3}, Lemma 
\ref{lm:2.2.2} and the fact that $\Vert v'\Vert \geq 
b\sqrt{\frac{\log{p}}{p}}$, we get 
\begin{equation}
	|\mathcal{F}_{r}(0,\sqrt{p}v')|\leq C 
	p^{3r/2}\exp \Big (-\frac{\varepsilon_{0}}{4} 
	b^{2}\log{p} \Big ),
\end{equation}
where the constant $C>0$ does not depend on $x\in U$, and the 
number $\varepsilon_{0}$ from Assumption \ref{item:beta} can be taken 
smaller than 
$1$.

Since we take $b > \sqrt{16k/\varepsilon_{0}}$, then for $r=0,\ldots, 
2k+1$, we get
\begin{equation}
	|\mathcal{F}_{r}(0,\sqrt{p}v')\kappa^{-1/2}(v')p^{-r/2}|\leq C 
	p^{-(2k-1)}.
	\label{eq:3.1.6}
\end{equation}
Finally, combining \eqref{eq:3.1.1}--\eqref{eq:3.1.6}, we get the 
first estimate as wanted for any $p>1$.

\noindent
\textbf{Step 2:} We next prove the second part of our theorem. For 
this purpose, we only need to consider sufficiently large $p$ such that 
$b\sqrt{\frac{\log{p}}{p}}\leq 
\frac{\varepsilon}{2}$, where $\varepsilon$ is given in Step 1.

In the expansion \eqref{eq:2.2.2}, we take $x_{0}=x, y=\exp_{x}(v'), 
N=1$, so $\mathrm{dist}(x,y)=\Vert v'\Vert =|z'|\leq 
b\sqrt{\frac{\log{p}}{p}}$, where $z'\in\C$ is the complex coordinate 
for $v'$. We infer
\begin{equation}
	\begin{split}
		B_{p}(x,y)=&\,p\kappa^{-1/2}(v')\frac{a(x)}{2\pi}\exp \Big(-\frac{1}{4}a(x)p\Vert v'\Vert ^{2}\Big )\\
		&+p^{1/2}\kappa^{-1/2}(v')\frac{a(x)}{2\pi}\exp \Big (-\frac{1}{4}a(x)p\Vert v'\Vert ^{2} \Big)\mathcal{J}_{1}(0,\sqrt{p}v')+\mathcal{O}(|\log{p}|^{4}).
	\end{split}
	\label{eq:3.1.7}
\end{equation}
Since $\Vert v'\Vert \leq b\sqrt{\frac{\log{p}}{p}}$, using \eqref{eq:Jr} we infer that
$|\mathcal{J}_{1}(0,\sqrt{p}v')|\leq C |\log{p}|^{3/2}$. The previous in combination with \eqref{eq:2.1.1} then supplies us with
\begin{equation}
	\begin{split}
		\frac{\exp(\frac{1}{4}a(x)p\Vert v'\Vert 
		^{2})B_{p}(x,y)}{\sqrt{B_{p}(x,x)}\sqrt{B_{p}(y,y)}}&=\frac{pa(x)\kappa^{-1/2}(v')}{\sqrt{B_{p}(x,x)}\sqrt{B_{p}(\exp_{x}(v'),\exp_{x}(v'))}}\\
		&\quad +\mathcal{O}(p^{-1/2}|\log{p}|^{3/2} + 
		p^{-1}|\log{p}|^{4})\\
		&=1+\mathcal{O}(\Vert v'\Vert +p^{-1/2}|\log{p}|^{3/2} + 
		p^{-1}|\log{p}|^{4})\\
		&=1+o(1), \text{ as }p\rightarrow +\infty.
	\end{split}
	\label{eq:3.1.8bis}
\end{equation}

Note that in the definition of $P_{p}$ we have the Hermitian norm of 
$B_{p}(x,y)$. Since in the asymptotic expansion \eqref{eq:3.1.7} we 
have trivialized the line bundle near $x$ using the Chern 
connections, we have
\begin{equation}
	|B_{p}(x,y)|_{h^{p}_{x}\otimes 
	h^{p,*}_{y}}=B_{p}(x,y)+\mathcal{O}(\Vert v'\Vert ).
	\label{eq:3.1.9bis}
\end{equation}
Combining \eqref{eq:3.1.8bis} and \eqref{eq:3.1.9bis}, we get the 
estimate \eqref{eq:1.5.2July} by taking the term $G_{p}(x,y)$ to be 
the $o(1)$-term in the last equation in \eqref{eq:3.1.8bis}. This completes the proof of Theorem 
\ref{thm:1.3.1}.

\section{Proofs of our results for punctured Riemann 
surfaces}\label{section3DLM}
In the sequel, we adopt the following notation and conventions: for positive functions $f, g: \, \mathbb N \to \R,$ we write $f(p)\lesssim g(p)$ if there 
exists a constant $C>0$ (possibly depending on 
some given data) such that $f(p)\leq C g(p)$ for all (sufficiently 
large) $p \in \mathbb N.$ Similarly, we write 
$f(p)\gtrsim g(p)$ if $f(p)\ge c g(p)$ for some constant $c > 0$ and 
all (sufficiently 
large) $p \in \mathbb N.$ Moreover, we write $f(p)\simeq g(p)$ if both 
$f(p)\lesssim g(p)$ and $f(p)\gtrsim g(p)$ hold.

Since the computations of this section are also 
applicable in 
the higher dimensional case and for a relatively compact open subset 
$U$ as 
described in Subsection \ref{section1.6} and Section 
\ref{section:higherdim}, we will always emphasize the quantity $d_{p}$ 
appearing in various estimates of this section. We use the punctured 
surface $\Sigma$ as an important example for the case $d_{p}\simeq p$. Another 
advantage of $\Sigma$ is that due to the work of \cite{AMM:20} we 
can study the divisors in the open subset $U\subset \Sigma$ which is not relatively compact.

\subsection{Supremum of norm of random holomorphic 
sections}\label{section3.1DLM}

As introduced above, $s_{p}$ will denote a random section with probability 
measure $\Upsilon_{p}$. Then $\mathcal{M}^{U}_{p}(s_{p})$ is a 
positive random variable. In this subsection, we study the expectation
$\mathbb{E}[\mathcal{M}^{U}_{p}(s_{p})]$ to understand the {\it typical} 
value of $\mathcal{M}^{U}_{p}(s_{p})$.

For a vector $\eta^{p}=(\eta^{p}_{1},\ldots,\eta^{p}_{d_{p}})\in \C^{d_{p}}$, set 
$\Vert \eta^{p}\Vert ^{2}=\sum_{j=1}^{d_{p}}|\eta^{p}_{j}|^{2}$. For $x\in 
U$, and for $s_{p}=\sum_{j\in O_{p}} \eta^{p}_{j}S^{p}_{j}$, we 
have
\begin{equation}
	|s_{p}(x)|_{h^{p}}\leq \Vert \eta^{p}\Vert B_{p}(x,x)^{1/2}.
	\label{eq:5.2.1k}
\end{equation}

The bounds in \eqref{eq:5.1.3} also give the bounds for 
$\sigma_{p}^{2}$, $p\in \mathbb{N}$.
\begin{lemma}\label{lm:5.1.2}
	There exists a constant $K_{0}>0$ such that for $p\in \mathbb{N}$,
	\begin{equation}
		\sigma_{p}^{2}\leq K_{0}d_{p}^{2}.
		\label{eq:5.1.5new}
	\end{equation}
\end{lemma}
\begin{proof}
	Since $\mathbb{E}[\eta^{p}_{j}]=0$, then 
	$\sigma^{2}_{p}=\mathbb{E}[|\eta^{p}_{j}|^{2}]$. Let $p$ be 
	sufficiently large such that $d_{p}>2$. By Jensen's 
	inequality, we get
	\begin{equation}
		\mathbb{E}[|\eta^{p}_{j}|^{2}]^{d_{p}/2}\leq 
		\mathbb{E}[|\eta^{p}_{j}|^{d_{p}}].
		\label{eq:5.1.6new}
	\end{equation}
	Then \eqref{eq:5.1.5new} follows from the assumption 
	\eqref{eq:5.1.3} for $\eta^{p}_{j}$. 
\end{proof}

\begin{lemma}
	We have the following inequalities of moments of $\Vert \eta^{p}\Vert $ 
	for $p\in\mathbb{N}$ sufficiently large,
	\begin{equation}
		\begin{split}
			& c_{0}d_{p}\leq \mathbb{E}[\Vert \eta^{p}\Vert ^{2}]\leq K_{0} 
			d^{3}_{p},\\
			& \mathbb{E}[\Vert \eta^{p}\Vert ^{d_{p}}] \leq C_{0} 
			(d_{p})^{2d_{p}+1}.
		\end{split}
		\label{eq:5.2.2k}
	\end{equation}
	Therefore, we have the lower bound estimate for all $p\in\mathbb{N}$ sufficiently large
	\begin{equation}
		\mathbb{E}[\Vert \eta^{p}\Vert ]\geq C_{0}^{-\frac{1}{d_{p}-2}} 
		(d_{p})^{-\frac{2d_{p}+1}{d_{p}-2}}(c_0 
		d_{p})^{\frac{d_{p}-1}{d_{p}-2}}\gtrsim 
		(d_{p})^{-1-\frac{4}{d_{p}-2}}.
		\label{eq:5.2.2bis}
	\end{equation}
\end{lemma}
\begin{proof}
	Note that by the assumption (cf.\ \eqref{eq:5.1.1}) for 
	$\eta^{p}_{j},$  $1 \le j \le d_p,$, we have
	\begin{equation}
		\mathbb{E}[\Vert \eta^{p}\Vert ^{2}]=d_{p}\sigma^{2}_{p}.
		\label{eq:5.2.4bonn}
	\end{equation}
	Then the first inequality in \eqref{eq:5.2.2k} follows directly from 
	\eqref{eq:5.1.1} and \eqref{eq:5.1.5new}, we now prove the second one. 
	
	For $p$ 
	sufficiently large, we have $d_{p}>3$, set 
	\begin{equation}
		q=\frac{1}{1-\frac{2}{d_{p}}}>1.
		\label{eq:5.2.3k}
	\end{equation}
	Then 
	\begin{equation}
		q\frac{d_{p}}{2}\leq d_{p}.
		\label{eq:5.2.4k}
	\end{equation}
	
	By H\"{o}lder's inequality, we have
	\begin{equation}
		\Vert \eta^{p}\Vert ^{d_{p}}\leq \Big (\sum_{j\in O_{p}} 
		|\eta^{p}_{j}|^{d_{p}} \Big) (d_{p})^{q\frac{d_{p}}{2}}.
		\label{eq:5.2.5k}
	\end{equation}
	Then the second inequality in \eqref{eq:5.2.2k} follows directly 
	from \eqref{eq:5.1.3}, \eqref{eq:5.2.4k} and \eqref{eq:5.2.5k}.
	
	Set $a=\frac{d_{p}-2}{d_{p}-1}$, 
	$p_{1}=1/a=\frac{d_{p}-1}{d_{p}-2}$, $p_{2}=d_{p}-1$, then by 
	H\"{o}lder's inequality for the pair $(p_{1},p_{2})$, we get
	\begin{equation}
		\begin{split}
			\mathbb{E}[\Vert \eta^{p}\Vert ^{2}]&=\mathbb{E}[\Vert \eta^{p}\Vert ^{a}\Vert \eta^{p}\Vert ^{2-a}]\\
			&\leq 
			\mathbb{E}[\Vert \eta^{p}\Vert ]^{1/p_{1}}\mathbb{E}[\Vert \eta^{p}\Vert ^{d_{p}}]^{1/p_{2}}.
		\end{split}
		\label{eq:5.2.5bis}
	\end{equation}
	Using \eqref{eq:5.2.2k}, the inequality \eqref{eq:5.2.2bis} 
	follows. This completes our proof. 
\end{proof}

By \eqref{eq:2.1.3}, there is a constant $p_{0}\in\mathbb{N}$ 
(independent of open set $U$) such 
that for all $p>p_{0}$,
\begin{equation}
	\sup_{x\in U} B_{p}(x,x)< \frac{1}{\pi\sqrt{2\pi}}p^{3/2}.
	\label{eq:5.2.6}
\end{equation}
Then we have
\begin{equation}
	\mathcal{M}^{U}_{p}(s_{p})\leq \Vert \eta^{p}\Vert p^{3/4}.
	\label{eq:5.2.7}
\end{equation}

\begin{proposition}\label{prop:5.2.2}
	We have the following inequalities, for sufficiently large $p$,
	\begin{equation}
		\begin{split}
			& \mathbb{E}[\mathcal{M}^{U}_{p}(s_{p})]\leq 
			(K_{0})^{1/2} p^{3/4} d^{3/2}_{p}\lesssim p^{9/4},\\
			& \mathbb{E}[\mathcal{M}^{U}_{p}(s_{p})^{d_{p}}] \leq C_{0} 
			(d_{p})^{2d_{p}+1} p^{3d_{p}/4}\lesssim (Cp)^{Cd_{p}}, 
			\mathrm{\;for\;some\;constant\;}C>0.
		\end{split}
		\label{eq:5.2.9k}
	\end{equation}
	Moreover,
	\begin{equation}
		\mathbb{E}[\mathcal{M}^{U}_{p}(s_{p})]\gtrsim p^{-2}
		\label{eq:5.2.10bis}
	\end{equation}
\end{proposition}
\begin{proof}
	The inequalities in \eqref{eq:5.2.9k} follow directly from 
	$d_{p}\simeq p$,
	\eqref{eq:5.2.2k}, \eqref{eq:5.2.7} and 
	\begin{equation}
		\mathbb{E}[\Vert \eta^{p}\Vert ]^{2}\leq\mathbb{E}[\Vert \eta^{p}\Vert ^{2}].
		\label{eq:5.2.10}
	\end{equation}
	Now we prove the lower bound in \eqref{eq:5.2.10bis}. We fix a 
	point $x_{0}\in U$. Then
	\begin{equation}
		|s_{p}|_{h^{p}}^{2}(x_{0})=\sum_{j,l\in 
		O_{p}}\eta^{p}_{j}\overline{\eta}^{p}_{l}h^{p}(S^{p}_{j}(x_{0}), 
		S^{p}_{l}(x_{0})),
		\label{eq:5.2.11s}
	\end{equation}
	and using \eqref{eq:5.1.1} we infer that
	\begin{equation}
		\mathbb{E}[|s_{p}(x_{0})|_{h^{p}}^{2}]=\sum_{j\in 
		O_{p}}\mathbb{E}[|\eta^{p}_{j}|^{2}]|S^{p}_{j}(x_{0})|^{2}_{h^{p}}=\sigma^{2}_{p}B_{p}(x_{0},x_{0})\geq c_{0} B_{p}(x_{0},x_{0}).
		\label{eq:5.2.12s}
	\end{equation}
	The second inequality of \eqref{eq:5.2.9k} implies
	\begin{equation}
		\mathbb{E}[|s_{p}(z_{0})|_{h^{p}}^{d_{p}}] \leq C_{0} 
		(d_{p})^{2d_{p}+1} p^{3d_{p}/4}.
		\label{eq:5.2.13s}
	\end{equation}
	By the H\"{o}lder's inequality as in \eqref{eq:5.2.5bis}, we get
	\begin{equation}
		\mathbb{E}[|s_{p}(x_{0})|_{h^{p}}]\geq 
		\big(c_{0}B_{p}(x_{0},x_{0})\big)^{\frac{d_{p}-1}{d_{p}-2}}\Big(\frac{1}{C_{0} 
		(d_{p})^{2d_{p}+1} 
		p^{3d_{p}/4}}\Big)^{\frac{1}{d_{p}-2}}\gtrsim 
		d_{p}^{-\frac{7/4+2/d_{p}}{1-2/d_{p}}}.
		\label{eq:5.2.14}
	\end{equation}
	Since $d_{p}\simeq p$ and $\mathbb{E}[|s_{p}(x_{0})|_{h^{p}}]\leq 
	\mathbb{E}[\mathcal{M}^{U}_{p}(s_{p})]$, we get 
	\eqref{eq:5.2.10bis}. This finishes our proof.
\end{proof}

\begin{remark}\label{rm:3.1.4paris}
	Note that the lower bound $p^{-2}$ in $\eqref{eq:5.2.10bis}$ is 
	clearly non-optimal. By \eqref{eq:5.2.14}, we get the following limit,
	\begin{equation}
		\liminf_{p\rightarrow 
		\infty}p^{7/4}\mathbb{E}[\mathcal{M}^{U}_{p}(s_{p})] \geq 
		c_{0}a(x_{0})>0,
		\label{eq:5.2.18}
	\end{equation}
	where we always have $a(x_{0})\geq 
	\varepsilon_{0}>0$. Since $x_{0}\in U$ is arbitrarily 
	chosen, then we get
	\begin{equation}
		\liminf_{p\rightarrow 
		\infty}p^{7/4}\mathbb{E}[\mathcal{M}^{U}_{p}(s_{p})] \geq 
		c_{0}\sup_{x\in U}a(x).
		\label{eq:5.2.19}
	\end{equation}
\end{remark}

\begin{remark}
	If we take a relatively compact open subset $U$ in $\Sigma$, then 
	the estimates in Proposition \ref{prop:5.2.2} can be improved as 
	follows, as $p$ sufficiently large,
	\begin{equation}
		\begin{split}
			p^{-3/2-\delta}\lesssim \mathbb{E}[\mathcal{M}^{U}_{p}(s_{p})]\lesssim p^{2},
		\end{split}
		\label{eq:5.2.20}
	\end{equation}
	where $\delta>0$ is any sufficiently small number.
\end{remark}

Applying the Chebyshev inequality to the second inequality in 
\eqref{eq:5.2.9k}, we get the following result.
\begin{corollary}\label{cor:3.1.6}
	There exists a constant $C>0$, such that for any sequence 
	$\{\lambda_{p}\}_{p\in \mathbb{N}}$  of strictly positive 
	numbers, we have
	\begin{equation}
		\begin{split}
			\Upsilon_{p}(\big\{s_{p}\;:\; 
			\mathcal{M}^{U}_{p}(s_{p})\geq \lambda_{p}\big\})\lesssim 
			e^{-d_{p}\log\lambda_{p}+Cd_{p}\log p}.
		\end{split}
		\label{eq:5.2.21}
	\end{equation}
\end{corollary}

\subsection{Uniform bound on the marginal density function}\label{section3.2DLM}
In this subsection, we prove an 
important consequence of \eqref{eq:5.1.2}, i.e., an upper bound on the marginal densities of $\Upsilon_{p}$. We 
now fix a $p\in \mathbb{N}$. Let $V\subset \C^{d_{p}}$ be a 
$\C$-subspace of dimension $n\leq d_{p}$, and let $V^{\perp}\subset 
\C^{d_{p}}$ denote its orthogonal subspace with respect to the 
standard Hermitian metric on $\C^{d_{p}}$. 

If $v\in \C^{d_{p}}$, let 
$v=v_{0}+v_{1}$, $v_{0}\in V, v_{1}\in V^{\perp}$ denote the 
orthogonal decomposition of $v$. Let $\mathrm{dV}_{0}$, $\mathrm{dV}_{1}$ 
denote the standard Lebesgue volume elements on $V$, $V^{\perp}$ 
respectively such that
\begin{equation}
	\mathrm{dVol}_{p}(v)=\mathrm{dV}_{0}(v_{0})\mathrm{dV}_{1}(v_{1}).
	\label{eq:5.1.6}
\end{equation}

\begin{proposition}\label{prop:boundmarginal}
	For $v_{0}\in V$, set
	\begin{equation}
		g^{p}_{V}(v_{0})=\int_{v_{1}\in V^{\perp}} 
		f^{p}(v_{0}+v_{1}) \, \mathrm{dV}_{1}(v_{1}).
		\label{eq:5.1.7}
	\end{equation}
	Then $g^{p}_{V}$ is a probability density function on 
	$V$ such that
	\begin{equation}
		\sup_{v_{0}\in V}g^{p}_{V}(v_{0})\leq 
		M_{0}^{n}\binom{d_{p}}{n},
		\label{eq:5.1.8}
	\end{equation}
	where $M_{0}$ is the constant in \eqref{eq:5.1.2}, and $\binom{d_{p}}{n}=\frac{d_{p}!}{n!(d_{p}-n)!}\,\cdot$
\end{proposition}
\begin{proof}
	If $n=d_{p}$ or $0$, the proposition trivially holds true. Hence, without loss of generality we can and do assume $n<d_{p}$ from now on for the rest of the proof. Since $p$ is fixed, we simply set $d=d_{p}, k=d-n>0,$ and we let $E_{1},\ldots, E_{k}$ be an orthonormal 
	basis of $V^{\perp}$. Writing $e_{1}, \ldots, e_{d}$ for the 
	standard orthonormal basis of $\C^{d}$, this corresponds exactly to
	the sections $S^{p}_{j}$ under the identification $H^{0}_{(2)}(\Sigma, 
	L^{p})\simeq \C^{d_{p}}$. Write for $i=1,\ldots,k$,
	\begin{equation}
		E_{i}=\sum_{j=1}^{d} a_{i}^{j}e_{j}, \; a_{i}^{j}\in\C.
		\label{eq:5.1.9}
	\end{equation}
	
	Let $W_{p}$ denote the matrix $(a_{i}^{j})$ of size $k\times d$, 
	and denote by $W_{p}^{*}$ its complex adjoint matrix.
	The orthonormality of the basis implies
	\begin{equation}
		W_{p}W_{p}^{*}=\mathrm{Id}_{k\times k}.
		\label{eq:5.1.10}
	\end{equation}
	
	Let $I(d,k)$ denote all subsets of $\{1,\ldots,d\}$ of cardinality 
	$k$, then $|I(d,k)|=\binom{d}{n}$. If $S\in I(d,k)$, let 
	$W_{p,S}$ denote the square matrix consisting of the $k$ columns of 
	$W_{p}$ indexed by $S$ (in the order induced by $S$), and let $W_{p}^{*,S}$ 
	denote the square matrix consisting of $k$ rows of $W^{*}_{p}$ 
	indexed by $S$ (in the order induced by $S$). It is clear that $W_{p}^{*,S}$ 
	is exactly the complex adjoint matrix of $W_{p,S}$. Then, due to the 
	Cauchy-Binet formula (i.e., a generalized 
	Pythagorean or Gougu theorem), we have
	\begin{equation}
		1=\det W_{p}W_{p}^{*}=\sum_{S\in I(d,k)}\det 
		W_{p,S}W_{p}^{*,S}.
		\label{eq:5.1.11}
	\end{equation}
	
	Now observe that $\det 
	W_{p,S}W_{p}^{*,S}\geq 0,$ and hence due to \eqref{eq:5.1.11} there exists 
	$S_{V}\in I(d,k)$ such that
	\begin{equation}
		\det 
		W_{p,S_{V}}W_{p}^{*,S_{V}}\geq \frac{1}{|I(d,k)|}\,\cdot
		\label{eq:5.1.12}
	\end{equation}
	In particular, $W_{p,S_{V}}$ is an invertible square matrix.
	
	We now prove \eqref{eq:5.1.8}. Let $t=(t_{1},\ldots, t_{k})$ denote 
	the complex coordinates of $V^{\perp}$ with respect to the basis 
	$E_{i}$, $i=1,\ldots,k$. Then we can write
	\begin{equation}
		v_{1}=\sum_{i=1}^{k}t_{i}E_{i}\in V^{\perp}, \; 
		\mathrm{dV}_{1}(t)=\prod_{i=1}^{k}\frac{\sqrt{-1}}{2}dt_{i}\wedge 
		d\overline{t}_{i}.
		\label{eq:5.1.13}
	\end{equation}
	Let $s=(s_{j})_{j\in S_{V}}\in \C^{k}$ be another complex coordinate system 
	of $V^{\perp}$ such that $s_{j}=\sum_{i=1}^{k} t_{i}a^{j}_{i}$. 
	Then $W_{p,S_{V}}$ represents exactly the Jacobian matrix for the 
	holomorphic coordinate change from $t$ to $s$, so that the Jacobian 
	determinant for the real coordinate change is given by $\det 
	W_{p,S_{V}}W_{p}^{*,S_{V}}$. Then for any integrable function $F$ on $\C^{k}$, we have 
	\begin{equation}
		\int_{t\in\C^{k}}F\Big(\sum_{i=1}^{k} t_{i}a^{j}_{i},j\in 
		S_{V}\Big)\, \mathrm{dV}_{1}(t)=\int_{s\in \C^{k}}F(s)\frac{1}{\det 
		W_{p,S_{V}}W_{p}^{*,S_{V}}} \, \mathrm{dV}_{1}(s).
		\label{eq:5.1.13bis}
	\end{equation}	
	Writing $v_{0}=(v_{0}^{1},\ldots, v_{0}^{d})\in V\subset 
	\C^{d}$ we infer that
	\begin{equation}
		\begin{split}
f^{p}(v_{0}+v_{1})&= \prod_{j\notin 
S_{V}}f^{p}_{j}\Big(v_{0}^{j}+\sum_{i=1}^{k}t_{i}a^{j}_{i}\Big)\cdot 
\prod_{j\in S_{V}}f^{p}_{j}\Big(v_{0}^{j}+\sum_{i=1}^{k}t_{i}a^{j}_{i}\Big)\\
&\leq M_{0}^{n}\cdot \prod_{j\in 
S_{V}}f^{p}_{j}\Big(v_{0}^{j}+\sum_{i=1}^{k}t_{i}a^{j}_{i}\Big),
		\end{split}
		\label{eq:5.1.14}
	\end{equation}
	where the last inequality follows from \eqref{eq:5.1.2}.
	
	Now, applying the formula \eqref{eq:5.1.13bis}, we deduce
	\begin{equation}
		g^{p}_{V}(v_{0})\leq \frac{M_{0}^{n}}{\det 
		W_{p,S_{V}}W_{p}^{*,S_{V}}} \int_{s\in \C^{k}} \prod_{j\in 
		S_{V}}f^{p}_{j}(v_{0}^{j}+s_{j})\, \mathrm{dV}_{1}(s)
		=\frac{M_{0}^{n}}{\det 
		W_{p,S_{V}}W_{p}^{*,S_{V}}}\,\cdot
		\label{eq:5.1.16}
	\end{equation}
	Combining \eqref{eq:5.1.12} with \eqref{eq:5.1.16}, we get 
	\eqref{eq:5.1.8}. This completes our proof.
\end{proof}

\subsection{Proof of Theorem \ref{thm:1.3.2}}\label{section3.3DLM}
Note that the inequality \eqref{eq:1.4.2paris} follows from 
Proposition \ref{prop:5.2.2}. By Remark \ref{rm:3.1.4paris}, the 
lower bound in \eqref{eq:1.4.2paris} can be improved to
$\frac{1}{C_{U}}p^{-7/4-\epsilon}$ for any given $\epsilon>0$.

Now we start to prove \eqref{eq:5.2.2}. Note 
that
\begin{equation}
	\Big\{s_{p}\;:\; \big|\log{\mathcal{M}^{U}_{p}(s_{p})}\big|\geq 
	\delta p\Big\}\subset \Big\{s_{p}\;:\; \mathcal{M}^{U}_{p}(s_{p})\geq 
	e^{\delta p}\Big\}\cup \Big\{s_{p}\;:\; \mathcal{M}^{U}_{p}(s_{p})\leq 
	e^{-\delta p}\Big\}.
\end{equation}
Upon choosing $\lambda_{p} = e^{\delta p}$ in \eqref{eq:5.2.21}
this entails by $d_{p}\sim p$,
\begin{equation}
	\Upsilon_{p}(\{s_{p}\;:\; \mathcal{M}^{U}_{p}(s_{p})\geq e^{ 
	\delta p}\})\leq e^{-C_{U,\delta}p^{2}}, \forall p\gg 0.
	\label{eq:53.2.2}
\end{equation}
Now we consider the probability of $\{\mathcal{M}^{U}_{p}(s_{p}) 
\le \lambda_{p}\}$ for arbitrary sequences 
$\{\lambda_{p}\}_{p\in \mathbb{N}}$  of positive 
numbers less than $1$. We claim that there exist constants $C>0,C'>0$ 
such that for $p\gg 0$,
\begin{equation}
	\Upsilon_{p}(\big\{s_{p}\;:\; \mathcal{M}^{U}_{p}(s_{p})\leq 
	\lambda_{p}\big\})\leq e^{Cd_{p}\log \lambda_{p} +C'd_{p}\log 
	p}.
	\label{eq:53.2.7}
\end{equation}
If we take $\lambda_{p}=e^{-\delta p}$ in \eqref{eq:53.2.7}, we get, 
with a constant $C_{U,\delta}>0$,
\begin{equation}
	\Upsilon_{p}(\{s_{p}\;:\; \mathcal{M}^{U}_{p}(s_{p})\leq e^{- 
	\delta p}\})\leq e^{-C_{U,\delta}p^{2}}, \forall p\gg 0;
	\label{eq:53.2.7bis}
\end{equation}
then inequality \eqref{eq:5.2.2} follows.

Therefore, in the sequel, we only focus on proving \eqref{eq:53.2.7}, 
which is clearly a more general statement than that we actually need.
For $U'\subset U$ be a smaller open subset which is relatively 
compact in $\Sigma$. Then we have
\begin{equation}
	\Upsilon_{p}(\{s_{p}\;:\; \mathcal{M}^{U}_{p}(s_{p})\leq 
	\lambda_{p}\})\leq \Upsilon_{p}(\{s_{p}\;:\; 
	\mathcal{M}^{U'}_{p}(s_{p})\leq \lambda_{p}\}).
	\label{eq:53.2.8}
\end{equation}
Fix a point $x_{0}\in U'$ and a $2$-cube $[-t,t]^{2}$ in 
$\mathbb{R}^{2}\simeq T_{x_{0}}\Sigma$. We choose $t>0$ sufficiently 
small so that $$F_{t}:=\exp_{x_{0}}([-t,t]^{2})\subset U',$$ 
and that
\begin{equation}
	\frac{1}{2}\Vert v-u\Vert \leq 
	\mathrm{dist}(\exp_{x_{0}}(v),\exp_{x_{0}}(u))\leq 
	2\Vert v-u\Vert ,\;\text{for all $v,u\in [-t,t]^{2}$}.
	\label{eq:53.2.10}
\end{equation}
Instead of proving directly \eqref{eq:53.2.7}, it is enough to prove 
the following estimate:
\begin{equation}
	\Upsilon_{p}(\{s_{p}\;:\; \mathcal{M}^{F_{t}}_{p}(s_{p})\leq 
	\lambda_{p}\})\leq e^{Cd_{p}\log \lambda_{p} +C'd_{p}\log 
	p},\: \text{for $p\gg 0$}.
	\label{eq:53.2.11}
\end{equation}
The uniform estimates in Theorem 
\ref{thm:1.3.1} hold for the open set $U'$. 
While our proof of \eqref{eq:53.2.11} is inspired by 
the arguments in \cite[Subsection 3.2]{SZZ}, some new 
computational techniques such as Proposition \ref{prop:boundmarginal} 
are needed since we are concerned with non-Gaussian ensembles of random variables.

For each $p>0$, we consider the lattice points
\begin{equation}
	\Gamma_{p}:= \Big \{(v_{1},v_{2})\in \mathbb{Z}^{2}\;:\; 
	|v_{j}|\leq \frac{t\sqrt{p}}{d} \Big \}.
	\label{eq:53.2.13}
\end{equation} 
and for $v\in\Gamma_{p}$ we define lattice points on the surface
by
\begin{equation}
	x^{p}_{v}=\exp_{x_{0}}\left(\frac{d}{\sqrt{p}}v\right)\in F_{t},
	\label{eq:53.2.12}
\end{equation}

The number of lattice points is given by
\begin{equation}
n_p:=\sharp\Gamma_{p}= \Big(2\Big[\frac{t\sqrt{p}}{d}\Big]+
1 \Big)^{2}=\frac{4t^{2}}{d^{2}}p+\mathcal{O}(\sqrt{p})\simeq d_{p}.
\label{eq:53.2.14}
\end{equation}
For $v\in\Gamma_{p}$ we fix some $\lambda_{v}\in L_{x^{p}_{v}}$ with 
$|\lambda_{v}|_{h}=1$ and set
\begin{equation}
\xi_{v}=\frac{\langle \lambda_{v}^{\otimes p}, 
s_{p}(x^{p}_{v})\rangle_{h^{p}}}{B_{p}(x^{p}_{v},x^{p}_{v})^{1/2}}\,\cdot
\label{eq:53.2.15}
\end{equation}
Then $\xi_{v}$ is a complex valued random variable.
By Theorem \ref{thm:2.1.1}, for $x^{p}_{v}\in F_{t}$, we have the following uniform 
estimate for $p\geq 1$ and $v\in\Gamma_{p}$,
\begin{equation}
	B_{p}(x^{p}_{v}, x^{p}_{v})= p \frac{a(x^{p}_{v})}{2\pi}+\mathcal{O}(1),
	\label{eq:53.2.15bis}
\end{equation}
Then \eqref{eq:53.2.11} follows from the claim below
\begin{equation}
	\Upsilon_{p}(\{\max_{v}|\xi_{v}|\leq \lambda_{p}\})\leq e^{Cd_{p}\log \lambda_{p} +C'd_{p}\log 
	p},\, \forall\, p\gg 0.
	\label{eq:53.2.16}
\end{equation}
Note that $s_{p}=\sum_{j}\eta^{p}_{j}S^{p}_{j}$, so we have
\begin{equation}
\xi_{v}=\sum_{j}\eta^{p}_{j}\frac{\langle \lambda_{v}^{\otimes p}, 
S^{p}_{j}(x^{p}_{v})\rangle_{h^{p}}}{B_{p}(x^{p}_{v},
x^{p}_{v})^{1/2}}\,\cdot	\label{eq:53.2.17}
\end{equation}
Recall that $\eta^{p}_{j}$, $j=1,\cdots,d_{p}$ are independently distributed random 
variables with expectation $\E[\eta^{p}_j]=0$ and uniformly bounded 
variance $\sigma_{p}^{2}$ as in \eqref{eq:5.1.1}, \eqref{eq:5.1.5new}.  Then 
\begin{equation}
	c_{0}\leq\mathbb{E}[|\xi_{v}|^{2}]=\sigma_{p}^{2}\leq K_{0} 
	d_{p}^{2}.
	\label{eq:53.2.18}
\end{equation}

Let $\Delta_{uv}:=\mathbb{E}[\xi_{u}\overline{\xi}_{v}]$ denote the covariance of 
$\xi_{u}$ and $\xi_{v}$, for $u,v\in\Gamma_{p}$, and let 
$\Delta=(\Delta_{uv})_{u,v\in\Gamma_{p}}$ denote the covariance 
matrix. Then by \eqref{eq:1.3.1}, \eqref{eq:53.2.17}, 
\begin{equation}
	\big|\mathbb{E}[\xi_{u}\overline{\xi}_{v}]\big|
	=\sigma^{2}_{p}P_{p}(x^{p}_{u},x^{p}_{v}).
	\label{eq:53.2.19}
\end{equation}
For $b=\sqrt{32/\varepsilon_{0}+1}$ 
we get by 
Theorem \ref{thm:1.3.1} that for $p\gg 0$,
\begin{equation}
\big|\Delta_{uv}\big|\leq\begin{cases} 
2\sigma_{p}^{2}\exp(-\frac{a(x^{p}_{u})p}{4}
\mathrm{dist}(x^{p}_{u},x^{p}_{v})^{2}) & \text{if } 
	\mathrm{dist}(x^{p}_{u},x^{p}_{v}) \leq b\sqrt{\frac{\log p}{p}}, \\
	\sigma_{p}^{2} \mathcal{O}(p^{-2})       & \text{if } 
	\mathrm{dist}(x^{p}_{u},x^{p}_{v}) \geq 
	b\sqrt{\frac{\log p}{p}},
\end{cases}
\label{eq:53.2.21}
\end{equation}
where the constant defining $\mathcal{O}(\cdot)$ is independent 
of $p$.

Fix $u\in \Gamma_{p}$, then by \eqref{eq:53.2.14} and the second 
estimate in \eqref{eq:53.2.21}, we have
\begin{equation}
	\frac{1}{\sigma^{2}_{p}}\sum_{v\in\Gamma_{p},v\neq u}\big|\Delta_{uv}\big|\leq 
	\sum_{\mathrm{near}}+\mathcal{O}(p^{-1}),
	\label{eq:53.2.22}
\end{equation}
where
\begin{equation}
	\sum_{\mathrm{near}}=\sum\Big\{\frac{1}{\sigma^{2}_{p}}\big|\Delta_{uv}\big|\;:\; 
	0<\mathrm{dist}(x^{p}_{u},x^{p}_{v})\leq b\sqrt{\tfrac{\log 
	p}{p}}\Big\}.
	\label{eq:53.2.23}
\end{equation}
Noting that
\begin{equation}
a(x^{p}_{u})\mathrm{dist}(x^{p}_{u},x^{p}_{v})^{2}>
\varepsilon_{0}\frac{d^{2}}{4p}\Vert u-v\Vert ^{2},
	\label{eq:53.2.24}
\end{equation}
 the first estimate in \eqref{eq:53.2.21} supplies us with
\begin{equation}
	\begin{split}
		\sum_{\mathrm{near}}&\leq 2\sum_{v\neq u} 
		e^{-\frac{\varepsilon_{0}d^{2}}{16}\Vert u-v\Vert ^{2}}
		\leq 2\sum_{v\in\mathbb{Z}^{2},v\neq 
		0}e^{-\frac{\varepsilon_{0}d^{2}}{16}\Vert v\Vert ^{2}}\\
		&\leq b'\int_{x\in\mathbb{R}^{2},\Vert x\Vert \geq 2/3} 
		e^{-\frac{\varepsilon_{0}d^{2}}{64}\Vert x\Vert ^{2}} \, {\rm d}x
		\leq b'' e^{-\frac{\varepsilon_{0}d^{2}}{144}}\leq 
		\frac{1}{3},
	\end{split}
	\label{eq:53.2.25}
\end{equation}
where the constants $b'$, $b''>0$ are independent of $p\gg 0$, and 
$d>0$ is chosen large enough so as to guarantee the last inequality.

We denote the $\ell^{\infty}$-norm of $\eta \in \C^{n}$
by $\Vert\eta\Vert_\infty=\max_{v}|\eta_{v}|$. Furthermore, we write
$I_n$ for
the $n\times n$ identity matrix, as well as $\Delta=\sigma^{2}_{p}I_n+A$, where $A$ has zero diagonal entries. 
By \eqref{eq:53.2.22} and \eqref{eq:53.2.25}, for $p\gg 0$, 
\begin{equation}
	\Vert A\eta\Vert_\infty\leq 
	\frac{\sigma^{2}_{p}}{2}\Vert\eta\Vert_\infty, \; \eta\in\C^{n}.
	\label{eq:53.2.26}
\end{equation}
Then
\begin{equation}
	\Vert\Delta\eta\Vert_\infty\geq 
	\Vert\sigma^{2}_{p}\eta\Vert_\infty - 
	\Vert A\eta\Vert_\infty\geq 
	\frac{\sigma^{2}_{p}}{2}\Vert\eta\Vert_\infty\geq
	\frac{c_{0}}{2}\Vert\eta\Vert_\infty.
	\label{eq:53.2.27}
\end{equation}
As a Hermitian square matrix, $\Delta$ is invertible and the eigenvalues of 
$\Delta^{-1}$ are bounded above by $2/\sigma^{2}_{p}$. Now set
\begin{equation}
	\zeta=(\zeta_{v}):=\Delta^{-1/2}\xi \in \C^{n},
	\label{eq:53.2.28}
\end{equation}
so that the coordinates $\zeta_{v}$ are random variables which are 
centered with finite variance, but generally they are not 
independently distributed. Moreover, we have
\begin{equation} \label{eq:5covVanish}
	\E[\zeta_u \overline{\zeta}_v] = \delta_{uv}, \quad \forall\, u,v \in \Gamma_p.
\end{equation}
Next, note that each $\zeta_v$ is a linear combination of the $(\eta^{p}_j),$ i.e.,
\begin{equation}
	\zeta_v = \sum_j \eta^{p}_j \beta_j(v),
	\label{eq:5.3.26}
\end{equation}
where $\beta_j(v) \in \C$ are constants. 
To apply directly Proposition \ref{prop:boundmarginal}, we normalize 
the random variables $\eta^{p}_{j},j=1,\cdots, d_{p}$ as follows,
\begin{equation}
	\tilde{\eta}^{p}_{j}=\frac{1}{\sigma_{p}}\eta^{p}_{j}.
	\label{eq:5.3.27}
\end{equation}
Then the PDF of $\tilde{\eta}^{p}_{j}$ on $\C$, with respect to the 
Lebesgue measure, is given by
\begin{equation}
	\tilde{f}^{p}_{j}(z)=f^{p}_{j}(\sigma_{p}z)\sigma_{p}^{2}.
	\label{eq:5.3.28}
\end{equation}
By \eqref{eq:5.1.2} and \eqref{eq:5.1.1}, we have that for all $p$, 
$j\in\{1,\cdots,d_{p}\}$,
\begin{equation}
	\sup_{z\in\C}|\tilde{f}^{p}_{j}(z)|\leq 
	M_{0}\sigma_{p}^{2}<\infty.
	\label{eq:5.3.29}
\end{equation}
As in \eqref{eq:5.1.5}, we denote by $\tilde{f}^{p}$ the joint probability density function 
of the random vector $(\tilde{\eta}^{p}_{j})_{j=1}^{d_{p}}\in 
\C^{d_{p}}$.
For $v\in \Gamma_{p}$, set 
\begin{equation}
	E_{v}=(\sigma_{p}\overline{\beta}_{1}(v),
	\ldots,\sigma_{p}\overline{\beta}_{d_{p}}(v))\in \C^{d_{p}}.
	\label{eq:5.3.30}
\end{equation}
By \eqref{eq:5covVanish}, $\{E_{v}\}_{v\in\Gamma_{p}}$ forms an 
orthonormal set in 
$\C^{d_{p}}$, let $V_{p}\subset \C^{d_{p}}$ denote the $\C$-subspace 
spanned by $\{E_{v}\}_{v\in\Gamma_{p}}$. Then
\begin{equation}
	n=\dim_{\C}V_{p}\leq d_{p}.
	\label{eq:5.3.31}
\end{equation}
Let $K_{p}$ denote the $d_{p}\times n$-matrix whose columns are just 
the column vectors $E_{v}$, $v\in\Gamma_{p}$. Then
\begin{equation}
	K_{p}^{*}K_{p}=I_n.
	\label{eq:5.3.32}
\end{equation}
Set $Q_{p}=K_{p}K_{p}^{*}$, then 
$Q_{p}$ is exactly the square matrix defining the orthogonal projection from $\C^{d_{p}}$ onto 
$V_{p}$ in $\C^{d_{p}}$. Let $V_{p}^{\perp}=\mathrm{Im}(1-Q_{p})$ be the orthogonal complement of 
$V_{p}$. We also identify the vector 
$(\zeta_{v})_{v\in\Gamma_{p}}\in \C^{n}$ with $\sum_{v\in\Gamma_{p}} 
\zeta_{v}E_{v}\in V_{p}$. 

Considering $\zeta=(\zeta_{v})_{v\in \Gamma_{p}}$ and 
$\tilde{\eta}^{p}=(\tilde{\eta}^{p}_{j})_{j\in O_{p}}$ as column vectors, then 
\eqref{eq:5.3.26} is equivalent to the relation
\begin{equation}
	Q_{p}\tilde{\eta}^{p}=K_{p}\zeta.
	\label{eq:5.3.33}
\end{equation}
As in \eqref{eq:5.1.7}, for $\zeta\in \C^{n}$, define
\begin{equation}
	g^{p}_{V_{p}}(\zeta)=\int_{\eta\in 
	V_{p}^{\perp}}\tilde{f}^{p}(K_{p}\zeta+\eta)\,\mathrm{dV}_{1}(\eta).
	\label{eq:5.3.34}
\end{equation}
Then $g^{p}_{V_{p}}$ is exactly the probability density function on 
$\C^{n}$ (with respect to the standard Lebesgue measure) for the 
random vector $(\zeta_{v})_{v\in\Gamma_{p}}$ defined in 
\eqref{eq:53.2.28}. By Proposition \ref{prop:boundmarginal} and 
\eqref{eq:5.3.29}, we get
\begin{equation}
	\sup_{\zeta\in \C^{n}}g^{p}_{V_{p}}(\zeta)\leq 
	(M_{0}\sigma_{p}^{2})^{n}\binom{d_{p}}{n},
	\label{eq:5.3.35}
\end{equation}
By \eqref{eq:53.2.28}, for $p\gg 0$,
\begin{equation}
	\max_{v}|\zeta_{v}|\leq \sqrt{\frac{2n}{\sigma^{2}_{p}}}\max_{v}|\xi_{v}|.
	\label{eq:53.2.30}
\end{equation}
As a consequence of the above, for $p\gg 0$,
\begin{equation}
	\begin{split}
\Upsilon_{p}\Big(\Big\{\max_{v}|\xi_{v}|\leq  \lambda_{p}\Big\}\Big)&\leq 
\Upsilon_{p}\Big(\Big\{\max_{v}|\zeta_{v}|\leq  
\sqrt{\frac{2n}{\sigma^{2}_{p}}}
\lambda_{p}\Big\}\Big)\\
		&\le \int_{(\zeta_{v})\in\C^{n}, |\zeta_{v}|\leq 
\sqrt{\frac{2n}{\sigma^{2}_{p}}}\lambda_{p}}
(M_{0}\sigma_{p}^{2})^{n}\binom{d_{p}}{n}\mathrm{dVol}(\zeta)\\
		&=(2\pi M_{0}n)^{n} \lambda^{2n}_{p}\binom{d_{p}}{n}.
	\end{split}
	\label{eq:3.3.37DLM}
\end{equation}
Note that in the above computations, $n\simeq d_{p}$, using 
$\binom{d_{p}}{n}\leq d_{p}!$, we get the desired inequality 
\eqref{eq:53.2.16}, thus \eqref{eq:53.2.7} holds. This completes our 
proof.
\begin{remark}
	By examining the proofs to \eqref{eq:5.2.21} and to 
	\eqref{eq:53.2.7}, Theorem \ref{thm:1.3.2} still holds if Condition 
	\eqref{eq:5.1.2} is replaced by a milder one: there exists 
	$k_{0}\in\bN$, $M_{0}>0$ such that for $p\gg 0$,
	\begin{equation}
		\sup_{z\in \C}|f^{p}_{j}(z)|\leq M_{0}\, p^{k_{0}}.
		\label{eq:3.3.39DLM}
	\end{equation}
\end{remark}


\subsection{Proof of Proposition \ref{prop:1.3.3}}\label{subs3.3}
We start by showing the integrability of the function $|\log|s_{p}|_{h^{p}}|$ 
for a nonzero $s_{p}\in 
H^{0}_{(2)}(\Sigma, L^{p})$ and for $p\geq2$ 
on each $V_{j}$ (as in Assumption \ref{item:alpha}) 
with respect to $\omega^{\Sigma}$. 
We just consider an open subset 
$\mathbb{D}^{*}_{r}\subset V_{j}$ for some $r\in\,]0,1[\,$. Let $1$ 
denote the canonical holomorphic frame of $L$ over 
$\mathbb{D}^{*}_{r}$ so that
\begin{equation}
	|1|^{2}_{h}(z)=|\log(|z|^{2})|.
	\label{eq:3.3.1bis}
\end{equation}
Then the section $s_{p}$, restricting on $\mathbb{D}^{*}_{r}$, can be 
written as
\begin{equation}
	s_{p}(z)=z^{k}f(z)1^{\otimes p}(z),
	\label{eq:3.3.2}
\end{equation}
where $f(z)$ is holomorphic function on $\overline{\mathbb{D}}$ with 
$f(0)\neq 0$, and $k\geq 1$ is the vanishing order of $s_{p}$ at 
$a_{j}$. Then for $z\in \mathbb{D}^{*}_{r}$, we have
\begin{equation}
	\log(|s_{p}|^{2}_{h^{p}})=2k\log|z|+\log(|f|^{2})+p\log|\log(|z|^{2})|.
	\label{eq:3.3.3bis}
\end{equation}
Note that
\begin{equation}
	\int_{0}^{r}\log{t}\frac{2tdt}{t^{2}\log^{2}{t}}=\infty.
	\label{eq:3.3.4bis}
\end{equation}
Comparing \eqref{eq:1.1.1}, the $\log|z|$-term in \eqref{eq:3.3.3bis} 
with \eqref{eq:3.3.4bis}, we get that $\log(|s_{p}|^{2}_{h^{p}})$ is not integrable with 
respect to the volume form $\omega_{\mathbb{D}^{*}}$ on 
$\mathbb{D}^{*}_{r}$.

As stated in Proposition \ref{prop:1.3.3}, we only consider a 
relatively compact open subset $U$.
Using instead Theorem \ref{thm:1.3.2}, the Proof of Proposition \ref{prop:1.3.3} follows exactly from the 
arguments in \cite[Subsection 4.1]{SZZ}. Here, we just summarize this 
proof briefly.

For $t>0$, we introduce the following notation
\begin{equation}
	\log^{+}{t}=\max\{\log{t},0\},\; 
	\log^{-}{t}:=\log^{+}(1/t)=\max\{-\log{t},0\}.
	\label{eq:3.3.4}
\end{equation}
Then
\begin{equation}
	|\log{t}|=\log^{+}{t}+\log^{-}{t}.
	\label{eq:3.3.5}
\end{equation}

Let $U$ be a relatively compact open subset in $\Sigma$, then for a 
nonzero $s_{p}$, $\big|\log|s_{p}|_{h^{p}}\big|$ in integrable on 
$\overline{U}$ with respect to $\omega_{\Sigma}$. We first to show the following claim:
\begin{equation}
	\Upsilon_{p}(\{s_{p}\;:\; 
	\int_{U}\log^{+}{|s_{p}|_{h^{p}}}\,\omega_{\Sigma}\geq 
	\frac{\delta}{2} p\})\leq e^{-C_{U,\delta}p^{2}}.
	\label{eq:3.3.6}
\end{equation}
Indeed, we have
\begin{equation}
	\log^{+}{|s_{p}|_{h^{p}}}\leq 
	|\log\mathcal{M}^{U}_{p}(s_{p})|.
	\label{eq:3.3.7}
\end{equation}
Then
\begin{equation}
	\begin{split}
		&\Upsilon_{p}\Big(\Big\{s_{p}\;:\; 
		\int_{U}\log^{+}{|s_{p}|_{h^{p}}}\,
		\omega_{\Sigma}\geq 
		\frac{\delta}{2} p\Big\}\Big)\\
		&\leq \Upsilon_{p}\Big(\Big\{s_{p}\;:\; 
		|\log\mathcal{M}^{U}_{p}(s_{p})|\geq 
		\frac{\delta}{2 \mathrm{Area}(U)} p\Big\}\Big),
	\end{split}
	\label{eq:3.3.9}
\end{equation}
where $\mathrm{Area}(U)$ denotes the area of $U$ with respect to 
$\omega_{\Sigma}$. Then \eqref{eq:3.3.6} follows exactly from Theorem \ref{thm:1.3.2}.

Next step is to prove that 
\begin{equation}
	\Upsilon_{p}\Big(\Big\{s_{p}\;:\; 
	\int_{U}\log^{-}{|s_{p}|_{h^{p}}}\,\omega_{\Sigma}\geq 
	\frac{\delta}{2} p\Big\}\Big)\leq e^{-C_{U,\delta}\,p^{2}},
	\label{eq:3.3.10}
\end{equation}
where we use \eqref{eq:3.3.6} and the property of sub-harmonic 
functions.
Suppose that $U$ contains an annulus $B(2,3):=\{z\in\C\;:\;2< 
|z|<3\}$ (after rescaling on the coordinate), and the line bundle $L$ 
on $B(1,4)$ (still contained in $U$) has a holomorphic local frame 
$e_{L}$. Set $\alpha(z)= 
\log|e_{L}(z)|^{2}_{h}$. For $s_{p}\in H^{0}_{(2)}(\Sigma, L^{p})$, 
we can write
\begin{equation}
	s_{p}=f_{p}e^{\otimes p}_{L},
	\label{eq:3.4.11DLM}
\end{equation}
where $f_{p}$ is a holomorphic function on $B(1,4)$. Then
\begin{equation}
	\log|s_{p}|_{h^{p}}=\log|f_{p}|+\frac{p}{2}\alpha.
	\label{eq:3.4.12DLM}
\end{equation}
In the following estimates, each $K_{\bullet}$ denotes a sufficiently 
large positive constant. Then by \eqref{eq:3.3.5} and \eqref{eq:3.3.9}, we have
\begin{equation}
	\Upsilon_{p}\Big(\Big\{s_{p}\;:\; 
	\int_{B(2,3)}\log^{+}{|f_{p}|}\,\omega_{\Sigma}\geq 
	K_{1} p \Big\} \Big)\leq e^{-C_{U,K_{1}}\, p^{2}},
	\label{eq:3.4.13DLM}
\end{equation}
Using the Poisson kernel and the sub-mean inequality for 
$\log(|f_{p}|)$, we improve \eqref{eq:3.4.13DLM} as follows,
\begin{equation}
	\Upsilon_{p}\Big(\Big\{s_{p}\;:\; 
	\int_{B(2,3)}\big|\log{|f_{p}|}\big|\,\omega_{\Sigma}\geq 
	K_{2} p\Big \} \Big)\leq e^{-C_{U,K_{2}}\, p^{2}},
	\label{eq:3.4.14DLM}
\end{equation}
From here, we proceed exactly as in \cite[Subsection 
4.1, pp. 1992]{SZZ}. As a consequence, for a $\delta\in\; 
]0,\frac{1}{2}]$, we get a finite set of points 
$\{z_{j}\}_{j=1}^{q}$ in $B(2,3)$ such that for all $s_{p}$
we have
\begin{equation}
	\begin{split}
		&-\int_{B(2,3)} \log|s_{p}|_{h_{p}}\omega_{\Sigma}\\
		&\quad\lesssim 
		-\sum_{j=1}^{q} \mu_{j} 
		\log|s_{p}(z_{j})|_{h^{p}}+K_{3}\delta\int_{B(2,3)} 
		\big|\log{|f_{p}|}\big|\,\omega_{\Sigma}+p\delta K_{3} \sup_{z\in 
		B(2,3)}|d\alpha(z)|_{\omega_{\Sigma}},
	\end{split}
\end{equation}
where the quantities $q$ and $\mu_{j}>0$ only depend on $\delta$, and 
$\sum_{j=1}^{q}\mu_{j}\simeq 1$.
Applying Theorem \ref{thm:1.3.2} to each term 
$\log|s_{p}(z_{j})|_{h^{p}}$ and taking advantage of \eqref{eq:3.4.13DLM}, we infer that
\begin{equation}
	\Upsilon_{p}\Big(\Big\{s_{p}\;:\; 
	-\int_{B(2,3)}\log{|s_{p}|_{h^{p}}}\,\omega_{\Sigma}\geq 
	K_{4}\delta p\Big\}\Big)\leq e^{-C_{U,\delta}p^{2}},\;\forall\;p\gg 0.
	\label{eq:3.4.16DLM}
\end{equation}

Note that $\log^{-}=-\log+\log^{+}$ and that a finite set of annulus 
$B(2,3)$ covers $U$, we get \eqref{eq:3.3.10} from 
\eqref{eq:3.3.9} and
\eqref{eq:3.4.16DLM}.

\begin{remark}
	Actually, one can see that, by the above sketched proof from 
	\cite[Subsection 4.1]{SZZ}, there is no absolute reason to take 
	the volume form $\omega_{\Sigma}$. If we use instead a smooth 
	volume form $\mathrm{d}\mu$ on $\overline{\Sigma}$, then a 
	probability estimate as in \eqref{eq:1.3.5} also holds even for 
	an open subset $U\subset \Sigma$ which is not relatively compact.
\end{remark}

\subsection{Proof of Theorem \ref{thm:1.3.5}}\label{subs3.4}

As explained in the last part of Subsection \ref{subs1.3}, 
we need to control the vanishing order of holomorphic sections 
in $H^{0}_{(2)}(\Sigma, L^{p})$ at the punctures.

For $a_{j}\in D$, set
\begin{equation}
	H^{0}(\overline{\Sigma}, L^{p})_{(a_{j},k)}:=\{s\in H^{0}(\overline{\Sigma}, L^{p})\;:\; 
	\mathrm{ord}_{a_{j}}(s)\geq k\}.
	\label{eq:3.4.1}
\end{equation}
It is clear that $H^{0}(\overline{\Sigma}, L^{p})_{(a_{j},k)}$ is a vector 
subspace of $H^{0}(\overline{\Sigma}, L^{p})$.
We always view $H^{0}_{(2)}(\Sigma, L^{p})$ as a subspace of 
$H^{0}(\Sigma, L^{p})$, set
\begin{equation}
	H^{0}_{(2)}(\Sigma, 
	L^{p})_{(a_{j},k)}=H^{0}_{(2)}(\Sigma,L^{p})\cap 
	H^{0}(\overline{\Sigma}, L^{p})_{(a_{j},k)}.
	\label{eq:3.4.2}
\end{equation}

\begin{lemma}\label{lm:3.4.1}
	There exist $p_{0}>0$, $k_{0}>0$ such that for any $a_{j}\in 
	D$, $\forall p\geq p_{0}$, $k\geq k_{0}$,
	\begin{equation}
		\dim H^{0}_{(2)}(\Sigma, 
		L^{p})_{(a_{j},k)} \leq d_{p}-1,
		\label{eq:3.4.3}
	\end{equation}
	so that
	\begin{equation}
		\Upsilon_{p}(H^{0}_{(2)}(\Sigma, 
		L^{p})_{(a_{j},k)})=0.
		\label{eq:3.4.4}
	\end{equation}
\end{lemma}
\begin{proof}
	It is clear that \eqref{eq:3.4.4} is a direct consequence of 
	\eqref{eq:3.4.3} since $\Upsilon_{p}$ has an integrable PDF on 
	$H^{0}_{(2)}(\Sigma,L^{p})$ with respect to the Lebesgue measure. We only need to prove \eqref{eq:3.4.3} for a fixed $a_{j}\in D$. Note that $L$ is a positive holomorphic line bundle on 
	$\overline{\Sigma}$ (since its degree is positive), so that for any 
	sufficiently large $p$, there exists a nonzero 
	section $s_{j,p}\in H^{0}(\overline{\Sigma},L^{p})$ such that 
	\begin{equation}
		s_{j,p}(a_{j})\neq 0.
		\label{eq:3.4.5}
	\end{equation}
We fix a sufficiently large $p_{0}\in\bN$ such that 
$H^{0}_{(2)}(\Sigma,L^{p_{0}})$ has 
a nonzero section $f_{p_{0}}$, and that if $p\geq 2p_{0}$, then
	\begin{equation}
		S_{p}:=f_{p_{0}}\otimes s_{j,p-p_{0}}\in H^{0}_{(2)}(\Sigma, 
		L^{p})
		\label{eq:3.4.6}
	\end{equation}
	has vanishing order
	\begin{equation}
		\mathrm{ord}_{a_{j}}(S_{p})=\mathrm{ord}_{a_{j}}(f_{p_{0}}).
		\label{eq:3.4.7}
	\end{equation}
	As a consequence, for $k\geq 
	k_{0}:=\mathrm{ord}_{a_{j}}(f_{p_{0}})+1$,
	\begin{equation}
		0\neq S_{p}\notin H^{0}_{(2)}(\Sigma, 
		L^{p})_{(a_{j},k)},
		\label{eq:3.4.8}
	\end{equation}
	so that \eqref{eq:3.4.3} holds.
\end{proof}
Fix a $k_{0}$ in Lemma \ref{lm:3.4.1}. For $p\geq p_{0}$, $k\geq k_{0}$, set
\begin{equation}
	A^{p}_{k}=\bigcup_{j} H^{0}_{(2)}(\Sigma, 
	L^{p})_{(a_{j},k)} \subset H^{0}_{(2)}(\Sigma, 
	L^{p}).
	\label{eq:3.4.9}
\end{equation}
Then $\Upsilon_{p}(A^{p}_{k})=0$. This way, we get Lemma 
\ref{lm:1.3.6} as mentioned in Subsection \ref{subs1.3}.

Let $\widetilde{U}$ be an open subset of $\overline{\Sigma}$, set 
$U=\widetilde{U}\backslash D\subset \Sigma$. Then for $s_{p}\in H^{0}_{(2)}(\Sigma, 
L^{p})\backslash A^{p}_{k_{0}}$, $p\geq p_{0}$, we have
\begin{equation}
	|\mathcal{N}^{U}_{p}(s_{p})-\mathcal{N}^{\widetilde{U}}_{p}(s_{p})|\leq k_{0}N,
	\label{eq:3.4.10}
\end{equation}
the difference comes from the zeros of $s_{p}$ at the 
punctures $a_{j}$, $j=1,\ldots, N$, that might be included in 
$\widetilde{U}$.

\begin{proof}[Proof of Theorem \ref{thm:1.3.5}]	
We may assume the function $\varphi$ does not vanish identically on 
$\overline{\Sigma}$. Set $M_{\varphi}=\max_{x\in\overline{\Sigma}}|\varphi(x)|>0$.

Let $V_{\varphi}\subset \overline{\Sigma}$ be an open neighborhood of 
	$D$ (with smooth boundary), on which $\varphi$ is locally constant. In particular, 
	$\partial\overline{\partial}\varphi|_{V_{\varphi}}\equiv 0$.
	Let $p'_{0}>0$ be an integer such that
	\begin{equation}
		\frac{k_{0}NM_{\varphi}}{p'_{0}}\leq\frac{\delta}{3}\,\cdot
		\label{eq:3.4.12bis}
	\end{equation}
The Poincar\'{e}-Lelong formula \cite[Theorem 2.3.3]{MM07}
asserts that we have in the sense of measures on $\overline{\Sigma}$, 
	\begin{equation}
		\frac{\sqrt{-1}}{\pi}\partial\overline{\partial}\log{|s_{p}|_{h^{p}}}=[\mathrm{Div}_{\overline{\Sigma}}(s_{p})]-pc_{1}(L,h).
		\label{eq:3.4.13}
	\end{equation}

Then
\begin{equation}
\begin{split}
&\left(\frac{1}{p}[\mathrm{Div}(s_{p})],\varphi\right)
-\int_{\Sigma}\varphi c_{1}(L,h)\\
&=\left( 
\frac{1}{p}[\mathrm{Div}_{\overline{\Sigma}}(s_{p})]-
c_{1}(L,h),\varphi\right) +\bigg( \frac{1}{p}\big([\mathrm{Div}_{\Sigma}(s_{p})]-
	[\mathrm{Div}_{\overline{\Sigma}}(s_{p})]\big),\varphi\bigg)\\
&=\frac{\sqrt{-1}}{p\pi}\int_{\Sigma}
\log{|s_{p}|_{h^{p}}}\,\partial\overline{\partial}\varphi+
\bigg( \frac{1}{p}\big([\mathrm{Div}_{\Sigma}(s_{p})]-
	[\mathrm{Div}_{\overline{\Sigma}}(s_{p})]\big),\varphi\bigg).
		\end{split}
		\label{eq:3.4.14}
	\end{equation}
By \eqref{eq:3.4.10} - \eqref{eq:3.4.12bis}, if $s_{p}\in H^{0}_{(2)}(\Sigma, 
	L^{p})\backslash A^{p}_{k_{0}}$, then
	\begin{equation}
		\bigg|\bigg( 
	\frac{1}{p}\big([\mathrm{Div}(s_{p})]-
	[\mathrm{Div}_{\overline{\Sigma}}(s_{p})]\big),
	\varphi\bigg)\bigg|\leq\frac{\delta}{3},\quad
	\text{for all $p\geq \max\{p_{0},p'_{0}\}$}.
		\label{eq:3.4.15}
	\end{equation}
Since $\omega_{\Sigma}$ is smooth on $\Sigma\backslash V_{\varphi}$ 
we can set
	\begin{equation}
S_{\varphi}=\max_{x\in \Sigma\backslash 
V_{\varphi}}\left|\frac{\sqrt{-1}\partial
\overline{\partial}\varphi(x)}{\omega_{\Sigma,x}}\right|.
		\label{eq:3.4.16}
	\end{equation}
	For the general case, we can and we may assume that $S_{\varphi}>0$. Then
	\begin{equation}
		\begin{split}
			\left|\frac{\sqrt{-1}}{p\pi}\int_{\Sigma}\log{|s_{p}|_{h^{p}}}\,
			\partial\overline{\partial}\varphi\right|\leq 
			\frac{S_{\varphi}}{p\pi}\int_{\Sigma\backslash V_{\varphi}}
			\big|\log{|s_{p}|_{h^{p}}}\big|\,\omega_{\Sigma}.
		\end{split}
		\label{eq:3.4.17}
	\end{equation}
Therefore, we get that for $p\gg 0$ the following holds:
	\begin{equation}
		\begin{split}
			&\Big\{s_{p}\;:\; 
			\Big|(\frac{1}{p}[\mathrm{Div}(s_{p})],\varphi)
			-\int_{\Sigma}\varphi c_{1}(L,h)\Big|>\delta \Big\}\\
			&\; \subset \Big\{s_{p}\;:\; 
			\frac{S_{\varphi}}{p\pi}\int_{\Sigma\backslash 
			V_{\varphi}}\big|\log{|s_{p}|_{h^{p}}}\big|\,\omega_{\Sigma} 
			>\frac{2}{3}\delta \Big\} \cup A^{p}_{k_{0}}.
		\end{split}
		\label{eq:3.4.18}
	\end{equation}
	Upon recalling that $\Upsilon_{p}(A^{p}_{k_{0}})=0$ and by 
	applying Proposition \ref{prop:1.3.3} to \eqref{eq:3.4.18} we get 
	\eqref{eq:3.4.11}. This completes the Proof of our theorem.
\end{proof}

\subsection{Proof of Theorem \ref{thm:1.2.1}}\label{subs3.5}
(a) In order to prove \eqref{eq:1.4.4DLM}
we have to show that for all $\varphi\in 
\mathcal{C}^{\infty}_{0}(\Sigma)$, we have that $\Upsilon$-a.s.,
\begin{equation} \label{eq:conv}
\lim_{p \to \infty} \Big(\frac{1}{p}[\mathrm{Div}(s_{p})],\varphi\Big) 
=\int_{\Sigma}\varphi c_{1}(L,h).
\end{equation}
While this is a folklore consequence of Theorem 
\ref{thm:1.3.5} in probability theory, we provide the short deduction 
here for the sake of completeness. Write 
$Y_{p}=(\frac{1}{p}[\mathrm{Div}(s_{p})],\varphi)$ as well as 
$Y=\int_{\Sigma}\varphi c_{1}(L,h)$. If there was no $\Upsilon$-a.s.\ 
convergence, then by dominated convergence for $Z:= 
\limsup_{p\to\infty} |Y_p-Y|$ there would exist $\delta > 0$ such 
that $\Upsilon(Z>\delta)>\delta.$ Choosing $N_\delta \in \mathbb N$ 
such that $\sum_{p \ge N_\delta} \Upsilon(|Y_p-Y| > \delta) < 
\delta/2$ (which is possible due to Theorem \ref{thm:1.3.5}) leads to a contradiction via $\Upsilon(Z>\delta) \le 
\sum_{p \ge N_\delta} \Upsilon(|Y_p - Y| > \delta) < \delta/2.$

(b) Note that for the open subset $U\subset \Sigma$, we have
\begin{equation}
	\frac{\mathrm{Area}^{L}(U)}{2\pi}=\int_{U}c_{1}(L,h)<+\infty.
	\label{eq:3.5.1}
\end{equation}
Here, we require no relative compactness for $U$.

Fix an arbitrary $\delta>0$, we choose $\psi_{1}$, $\psi_{2}\in 
\mathcal{C}^{\infty}(\overline{\Sigma})$ to be real-valued functions which takes 
constant values near $a_{j}\in D$ such that
\begin{equation}
	\begin{split}
		&0\leq \psi_{1}\leq\chi_{U}\leq \psi_{2}\leq 1,\\
		&\int_{\Sigma} \psi_{1}c_{1}(L,h)\geq 
		\frac{\mathrm{Area}^{L}(U)}{2\pi}-\delta,\\
		&\int_{\Sigma} \psi_{2}c_{1}(L,h)\leq 
		\frac{\mathrm{Area}^{L}(U)}{2\pi}+\delta,
	\end{split}
	\label{eq:3.5.2}
\end{equation}
where $\chi_{U}$ is the characteristic function of $U$ on 
$\overline{\Sigma}$.

We apply Theorem \ref{thm:1.3.5} for $\psi_{1}$, then for 
$s_{p}$ not in an exceptional set of probability less than 
$e^{-C_{\psi_{1},\delta}p^{2}}$, we get
\begin{equation}
\mathcal{N}^{U}_{p}(s_{p})\geq 
		(\mathrm{Div}_{\Sigma}(s_{p}),\psi_{1})
		\geq p\int_{\Sigma}\psi_{1}c_{1}(L,h)-p\delta
		\geq p\,\frac{\mathrm{Area}^{L}(U)}{2\pi}-(p+1)\delta.
	\label{eq:3.5.3}
\end{equation}

Similarly, if we proceed with $\psi_{2}$, we get that for 
$s_{p}$ not in an exceptional set of probability less than $e^{-C_{\psi_{2},\delta}p^{2}}$,
\begin{equation}
	\begin{split}
		\mathcal{N}^{U}_{p}(s_{p})\leq 
		p\,\frac{\mathrm{Area}^{L}(U)}{2\pi}+(p+1)\delta.
	\end{split}
	\label{eq:3.5.4}
\end{equation}
Part (b) follows by combining \eqref{eq:3.5.3} and \eqref{eq:3.5.4}.

\subsection{Proof of Proposition \ref{prop:1.2.3}}\label{subs3.6}

Since we are concerned with the Gaussian ensembles, using the fact that 
$\mathrm{Area}(\Sigma):=\int_{\Sigma}\omega_{\Sigma}<\infty$, 
the first part of Proposition \ref{prop:1.2.3} follows from the same 
arguments in \cite[Subsection 4.2.4]{SZZ}. As for \eqref{eq:1.2.5}, 
we need a refined estimate for the norm of a holomorphic section 
near the punctures, which is explained as follows.

For $k_{0}\geq 2$, the sections in $H^{0}_{(2)}(\Sigma, L^{k_{0}})$ 
are exactly the ones in $H^{0}(\overline{\Sigma}, L^{k_{0}})$ which 
vanish at every $a_{j}\in D$. Since the zeros of a nontrivial 
holomorphic section of $L^{k_{0}}$ are isolated points in 
$\overline{\Sigma}$, we get the existence of $r_{j}\in \;]0,\frac{1}{2}[$ 
and $\tau_{j}$ as wanted.

We may and we always rescale $\tau_{j}$ by a nonzero constant so that 
$\sup_{x\in \overline{\Sigma}} |\tau_{j}(x)|_{h^{k_{0}}}=1$. The 
following lemma is elementary.

\begin{lemma}\label{lm:3.6.1}
	For $r\in \,]0,r_{j}[\,$, set
	\begin{equation}
		b(r)=-\log(\inf_{z\in 
		\mathbb{D}(r,r_{j})}|\tau_{j}(z)|_{h^{k_{0}}})>0.
		\label{eq:3.6.1}
	\end{equation}
	Then there exists $C_{j}>0$ such that
	\begin{equation}
		b(r)\leq C_{j} |\log{r}|, \;r\in \;]0,r_{j}[\,.
		\label{eq:3.6.2}
	\end{equation}
\end{lemma}
\begin{proof}
	Locally, we can write for $z\in \mathbb{D}^{*}_{2r_{j}}$,
	\begin{equation}
		\tau_{j}(z)= z^{m_{j}}g(z)1^{\otimes k_{0}}(z),
		\label{eq:3.6.3}
	\end{equation}
	where $m_{j}\in\mathbb{N}_{\geq 1}$ is the vanishing order of 
	$\tau_{j}$ at $a_{j}$, and $g$ is a holomorphic function such 
	that $g(0)\neq 0$. 
Set $v_{j}=\inf_{z\in \overline{\mathbb{D}}_{r_{j}}}|g(z)|>0$, then
	\begin{equation}
		1\geq\inf_{z\in 
		\mathbb{D}(r,r_{j})}|\tau_{j}(z)|_{h^{k_{0}}}\geq 
		r^{m_{j}}v_{j}|\log(r^{2}_{j})|^{k_{0}/2}.
		\label{eq:3.6.4}
	\end{equation}
	Then \eqref{eq:3.6.2} follows easily.
\end{proof}

\begin{proof}[Proof of \eqref{eq:1.2.5}]
	Set
	\begin{equation}
		E^{pk_{0}}_{1}=\Vert \tau_{j}^{\otimes 
		p}\Vert ^{-1}_{\mathcal{L}^{2}}\tau_{j}^{\otimes p}\in 
		H^{0}_{(2)}(\Sigma, L^{pk_{0}}).
		\label{eq:3.6.5}
	\end{equation}
	Note that 
	\begin{equation}
		\Vert \tau_{j}^{\otimes 
		p}\Vert _{\mathcal{L}^{2}}\leq \mathrm{Area}(\Sigma)^{1/2}.
		\label{eq:3.6.6}
	\end{equation}
	We then complete $\{E^{pk_{0}}_{1}\}$ to an orthonormal basis $
	\{E^{pk_{0}}_{1}$, $E^{pk_{0}}_{2}$, $\ldots$, 
	$E^{pk_{0}}_{d_{pk_{0}}}\}$ of $H^{0}_{(2)}(\Sigma, L^{pk_{0}})$.
	
	Since $\Upsilon_{p}$ is defined by i.i.d. standard random complex 
	Gaussian variables, its dependence on the basis $O_{p}$ is 
	eliminated. The random section $s_{pk_{0}}$ is given by
	\begin{equation}
		s_{pk_{0}}=\sum_{j}\xi_{j} 
		E^{pk_{0}}_{j}=\xi_{1}E^{pk_{0}}_{1} + s'_{pk_{0}},
		\label{eq:3.6.7}
	\end{equation}
	where 
	$\xi_{j}$, $j=1,\ldots,d_{pk_{0}}$ are i.i.d. standard random complex 
	Gaussian.
	
	Set $\xi'=(\xi_{2},\ldots,\xi_{d_{pk_{0}}})\in 
	\C^{d_{pk_{0}}-1}$. Similar to \eqref{eq:5.2.1k} and 
	\eqref{eq:5.2.7}, we have for 
	$p\gg 0$ and $x\in \Sigma$,
	\begin{equation}
		|s'_{pk_{0}}(x)|_{h^{pk_{0}}}\leq C\Vert \xi'\Vert  p^{3/4},
		\label{eq:3.6.8}
	\end{equation}
	where the constant $C>0$ is independent of $p$ and $z$.
	By Lemma \ref{lm:3.6.1} and \eqref{eq:3.6.5}, for $r\in 
	\;]0,r_{j}[\,$, we have
	\begin{equation}
		\inf_{z\in 
		\mathbb{D}(r,r_{j})}|E^{pk_{0}}_{1}(z)|_{h^{pk_{0}}}\geq 
		\frac{e^{-pb(r)}}{\mathrm{Area}(\Sigma)^{1/2}}.
		\label{eq:3.6.9}
	\end{equation}	
	Set
	\begin{equation}
		t_{p}(r)=\frac{e^{-pb(r)}}{C\mathrm{Area}(\Sigma)^{1/2}p^{3/4}\sqrt{d_{pk_{0}}}}>0.
		\label{eq:3.6.10}
	\end{equation}
	Note that $\Vert \xi'\Vert \leq \sqrt{d_{pk_{0}}} \max_{j\geq 
	2}|\xi_{j}|$, then for $r\in\,]0,r_{j}[\,$, we have
	\begin{equation}
	\Big\{s_{pk_{0}}=\sum_{j}\xi_{j} E^{pk_{0}}_{j}\;:\; 
	|\xi_{1}|>1, |\xi_{j}|<t_{p}(r), j\geq 2\Big\}
	\subset\Big\{s_{pk_{0}}\;:\; 
	\mathcal{N}^{\mathbb{D}(r,r_{j})}_{pk_{0}}(s_{pk_{0}})=0\Big\}.
		\label{eq:3.6.11}
	\end{equation}
	Therefore, for $r\in \;]0,r_{j}[\,$, we have
	\begin{equation}
		\begin{split}
	\Upsilon_{pk_{0}}\big(\{s_{pk_{0}}\;:\; 
	\mathcal{N}^{\mathbb{D}(r,r_{j})}_{pk_{0}}(s_{pk_{0}})=0\} \big) 
	\geq e^{-1}\Big(\frac{t_{p}(r)^{2}}{2}\Big)^{d_{pk_{0}}-1}.
		\end{split}
		\label{eq:3.6.12}
	\end{equation}
By \eqref{eq:3.6.2} there exists $c_{j}>0$ such that for 
	$r\in \;]0,r_{j}[\,$, $p\gg 0$,
	\begin{equation}
		e^{-1}\Big(\frac{t_{p}(r)^{2}}{2}\Big)^{d_{pk_{0}}-1}\geq 
		e^{-c_{j}|\log{r}|p^{2}}.
		\label{eq:3.6.13}
	\end{equation}
	This completes our Proof of \eqref{eq:1.2.5}.
\end{proof}

\section{Cusp forms on hyperbolic surfaces of finite 
volume}\label{section4} 
We give an important example where our results 
apply. Let $\overline\Sigma$ be a compact Riemann surface 
of genus $g$ and consider a finite set 
$D=\{a_1,\ldots,a_N\}\subset\overline\Sigma$. 
We also denote by $D$ the divisor $\sum_{j=1}^Na_j$ and let 
$\mathscr{O}_{\overline{\Sigma}}(D)$ 
be the associated line bundle.
Let $K_{\overline\Sigma}$ be the canonical
line bundle of $\overline\Sigma$.
The following conditions are equivalent:
\\[3pt]\indent
(i) $\Sigma=\overline\Sigma\smallsetminus D$ admits a complete 
K\"ahler-Einstein metric $\omega_\Sigma$
with $\operatorname{Ric}_{\omega_\Sigma}=-\omega_\Sigma$,
\\[2pt]\indent
(ii) $2g-2+N>0$, 
\\[2pt]\indent
(iii) the universal cover of $\Sigma$ is the upper-half plane
$\mathbb{H}$,
\\[2pt]\indent
(iv) $L=K_{\overline\Sigma}\otimes
\mathscr{O}_{\overline\Sigma}(D)$ is ample.
\\[3pt]
This follows from the Uniformization Theorem \cite[Chapter IV]{fk} 
and the fact that the Euler characteristic of $\Sigma$ equals 
$\chi(\Sigma)=2-2g-N$ and the degree of $L$ is 
\[\deg L=2g-2+N=-\chi(\Sigma).\] 
If one of these equivalent conditions is satisfied,
the K\"ahler-Einstein metric 
$\omega_{\Sigma}$ is induced by the Poincar\'e metric on 
$\mathbb{H}$;
$(\Sigma,\omega_{\Sigma})$ and
the formal square root of $(L,h)$ satisfy conditions \ref{item:alpha}
and \ref{item:beta}, see \cite[Lemma 6.2]{AMM:20}. 
Theorem \ref{thm:1.2.1},
Corollary \ref{C:1.2.2} and Proposition \ref{prop:1.2.3} 
hence apply to this context.

Let $\Gamma$ be the Fuchsian group associated with the above 
Riemann surface
$\Sigma$, that is, $\Sigma\cong\Gamma\backslash\mathbb{H}$.
Then $\Gamma$ is a geometrically
finite Fuchsian group of the first kind, without elliptic elements. 
Conversely, if $\Gamma$ is such a group, then 
$\Sigma:=\Gamma\backslash\mathbb{H}$
can by compactified by finitely many points $D=\{a_1,\ldots,a_N\}$
into a compact Riemann surface
$\overline\Sigma$ such that the equivalent conditions (i)-(iv) above
are fulfilled.

The space $\mathcal{M}_{2p}^\Gamma$ of 
$\Gamma$-modular forms 
of weight $2p$ is by definition the space 
of holomorphic functions $f\in\mathscr{O}(\mathbb{H})$
satisfying the functional equation
\begin{equation}\label{eq:6.24}
	f(\gamma z)=(cz+d)^{2p}f(z),\quad z\in\mathbb{H},\:\: 
	\gamma=\begin{pmatrix}a&b\\c&d\end{pmatrix}\in\Gamma,
\end{equation}
and which extend holomorphically to the cusps of $\Gamma$ (fixed
points of the parabolic elements).
If $f\in\mathscr{O}(\mathbb{H})$ satisfies
\eqref{eq:6.24}, then 
$fdz^{\otimes p}\in H^0(\mathbb{H},K_\mathbb{H}^p)$
descends to a holomorphic section $\Phi(f)$ of 
$H^0(\Sigma,K^p_\Sigma)\cong H^0(\Sigma,L^p)$. 
By \cite[Propositions\,3.3,\,3.4(b)]{Mum77}, $\Phi$ induces 
an isomorphism
$\Phi:\mathcal{M}_{2p}^\Gamma
\to H^0\big(\overline\Sigma,L^p\big)$.

The subspace of $\mathcal{M}_{2p}^\Gamma$ consisting of
modular forms vanishing at the cusps is called the 
space of \emph{cusp forms} (Spitzenformen) of weight $2p$ 
of $\Gamma$, denoted by $\mathcal{S}_{2p}^\Gamma$.
The space of cusps forms is endowed with 
the Petersson scalar product
\[
\langle f,g\rangle:=
\int_U f(z)\overline{g(z)}(2y)^{2p}\,dv_{\mathbb{H}}(z),
\]
where $U$ is a fundamental domain for $\Gamma$ and
$dv_{\mathbb{H}}=\frac{1}{2} y^{-2}dx\wedge dy$ is the 
hyperbolic volume form.

Under the above isomorphism, $\mathcal{S}_{2p}^\Gamma$
is identified to 
the space
$H^0\big(\overline\Sigma,L^p\otimes
\mathscr{O}_{\overline\Sigma}(D)^{-1}\big)=
H^0\big({\overline\Sigma},K_{\overline\Sigma}^p
\otimes\mathscr{O}_{\overline\Sigma}(D)^{p-1}\big)$
of holomorphic sections of $L^p$ over $\overline\Sigma$
vanishing on $D$.

If we endow $K_\mathbb{H}$ with the Hermitian metric
induced by the Poincar\'e metric on $\mathbb{H}$, 
the scalar product of two elements 
$udz^{\otimes p}, vdz^{\otimes p}\in K^p_{\mathbb{H},z}$
is $\langle udz^{\otimes p}, vdz^{\otimes p}\rangle
=u\overline{v}(2y)^{2p}$.
Hence, the Petersson scalar product corresponds to the
$\mathcal{L}^{2}$ inner product of pluricanonical forms on $\Sigma$,
\[\langle f,g\rangle=
\int_{\Sigma} \langle\Phi(f),\Phi(g)\rangle\, \omega_{\Sigma}\,,
\quad f,g\in\mathcal{S}_{2p}^\Gamma\,. 
\] 
The isomorphism $\Phi$ gives thus an isometry 
(see also \cite[Section\,6.4]{CM11})
\begin{equation}\label{eq:6.25}
	\mathcal{S}_{2p}^\Gamma\cong
	H^0\big(\overline\Sigma,L^p\otimes\mathscr{O}_{\overline\Sigma}
	(D)^{-1}\big)\cong
	H^0_{(2)}(\Sigma,K_\Sigma^p)\cong H^0_{(2)}(\Sigma, L^p),
\end{equation}
where $H^0_{(2)}(\Sigma, L^p)$ is the space of holomorphic
sections of $L^p$ that are square-integrable with respect to 
the volume form $\omega_\Sigma$
and the metric $h^p$ on $L^p$, with $h$ introduced in \cite[Lemma 6.2]{AMM:20}. 
Moreover, $H^0_{(2)}(\Sigma,K_\Sigma^p)$ is the space
of $\mathcal{L}^{2}$-pluricanonical
sections with respect to the metric $h^{K^p_\Sigma}$
and the volume form $\omega_\Sigma$, where we denote by 
$h^{K_\Sigma}$ the Hermitian metric induced by 
$\omega_\Sigma$ on $K_\Sigma$. 

We thus identify the space of cusp forms 
$\mathcal{S}_{2p}^\Gamma$ to a subspace of holomorphic sections
of $L^p$ by \eqref{eq:6.25}.
\begin{corollary}    \label{crl_CrlIntro1}
	Let $\Gamma\subset\operatorname{PSL}(2,\R)$ be 
	a geometrically finite Fuchsian group of the first kind 
	without elliptic elements. 
	Then the assertions of Theorems \ref{thm:1.3.2} \& \ref{thm:1.2.1},
	Corollary \ref{C:1.2.2} and Proposition \ref{prop:1.2.3}
	hold for the zeros of cusp forms $s_p\in\mathcal{S}^\Gamma_{2p}$.
\end{corollary}

\section{Higher dimensional complex Hermitian 
manifolds}\label{section:higherdim}
In this section, we consider the extension of the above results to 
the noncompact complete complex Hermitian manifold of higher 
dimension. Our geometric setting is described in Subsection 
\ref{section1.6}. At 
first, we recall the Bergman kernel expansion under this setting.

By \cite[Theorems 4.2.1 \& 
6.1.1]{MM07}, the Bergman kernel expansions described in Section 
\ref{section2bergman}, for both on-diagonal and 
off-diagonal, still hold. More precisely, there exist coefficients $a_{r}\in C^{\infty}(X)$, 
$r\in\bN$ such that the following asymptotic expansion
\begin{equation}
	B_{p}(x,x)=\sum_{r=0}^{\infty} a_{r}(x)p^{m-r}
\end{equation}
holds for any $\mathcal{C}^{\ell}$-topology on compact sets of $X$. 
In particular, let $\dot{R}^{L}\in\mathrm{End}(T^{(1,0)}X)$ such that 
for $u,v\in T_{x}^{(1,0)}X$,
\begin{equation}
	R^{L}(u,v)=g^{TX}(\dot{R}^{L}u,v),
\end{equation}
then 
\begin{equation}
a_{0}(x)=\det\Big(\frac{\dot{R}^{L}}{2\pi}\Big)>\varepsilon_{1}^{m}.
\end{equation}
In particular, if $K\subset X$ is compact, then there exists 
$C_{K}>0$ such that for $p\gg 0$,
\begin{equation}
	\max_{x\in K}B_{p}(x,x)\leq C_{K}\, p^{m}.
	\label{eq:5.1.4paris}
\end{equation}

If $X$ is noncompact, then the existence of a complete metric 
$\omega$ with $iR^{L}> \varepsilon_{1} 
\omega$ ($\varepsilon_{1} >0$) is equivalent to saying that $iR^{L}$ 
defines a complete K\"{a}hler metric. Recall that the volume 
$\mathrm{Vol}^{L}_{2m}(\cdot)$ is defined in \eqref{eq:1.6.8DLM}. As 
in \cite[Corollary 2.2]{DMS12}, under the Assumption \eqref{eq:6.0.4}, 
we have
\begin{equation}
	0< \frac{1}{m!}\int_{X}c_{1}(L,h)^{m}\leq \liminf_{p\rightarrow 
	\infty} p^{-m}d_{p}< \infty.
\end{equation}
As a consequence, we get
\begin{equation}
	d_{p}\simeq p^{m}\;,\;\;\mathrm{Vol}^{L}_{2m}(X)<\infty.	
	\label{eq:5.1.2DLM}
\end{equation}
Then for any open subset $U\subset X$, we have 
$\mathrm{Vol}^{L}_{2m}(U)<\infty$.

Furthermore, the off-diagonal and near-diagonal expansions as in Proposition 
\ref{prop:2.2.1} also hold (with suitable change according to the 
dimension $m$). For a precise statement on the near-diagonal 
expansion, we need to introduce the the 
complex coordinates for the real tangent space $T_{x}X$, $x\in X$.

Fix a point $x\in X$. Let $\{\mathbf{f}_{j}\}_{j=1}^{m}$ be an orthonormal basis of 
$(T_{x}^{1,0}X, g_{x}^{TX}(\cdot,\overline{\cdot}))$ such that
\begin{equation}
	\dot{R}^{L}_{x}\,\mathbf{f}_{j}=\mu_{j}(x)\mathbf{f}_{j},
\end{equation}
where $\mu_{j}(x)$, $j=1, \ldots, m$ are the eigenvalues of 
$\dot{R}^{L}_{x}$. We have
\begin{equation}
	\mu_{j}(x)>\varepsilon_{1},\; 
	a_{0}(x)=\prod_{j=1}^{m}\frac{\mu_{j}(x)}{2\pi}.
\end{equation}
Set 
$\mathbf{e}_{2j-1}=\frac{1}{\sqrt{2}}(\mathbf{f}_{j}+\overline{\mathbf{f}}_{j})$, 
$\mathbf{e}_{2j}=\frac{\sqrt{-1}}{\sqrt{2}}(\mathbf{f}_{j}-\overline{\mathbf{f}}_{j})$, $j=1, \ldots, m$. Then they form an orthonormal basis of the (real) tangent vector space $(T_{x}X, g_{x}^{TX})$. Now we introduce the complex coordinate for $T_{x}X$. If $v=\sum_{j=1}^{2m} v_{j}\mathbf{e}_{j}\in T_{x}X$, we can write
\begin{equation}
	v=\sum_{j=1}^{m}(v_{2j-1}+\sqrt{-1}v_{2j})\frac{1}{\sqrt{2}} 
	\mathbf{f}_{j} + \sum_{j=1}^{m}(v_{2j-1}-\sqrt{-1}v_{2j})\frac{1}{\sqrt{2}} 
	\overline{\mathbf{f}}_{j}.
	\label{eq:5.1.8paris}
\end{equation}
Set $z=(z_{1},\ldots,z_{m})$ with $z_{j}=v_{2j-1}+\sqrt{-1}v_{2j}$, 
$j=1,\ldots, m$. We call $z$ the complex coordinate of $v\in T_{x}X$. 
Then by \eqref{eq:5.1.8paris},
\begin{equation}
	\frac{\partial}{\partial z_{j}} = \frac{1}{\sqrt{2}} 
	\mathbf{f}_{j} , \; \frac{\partial}{\partial \overline{z}_{j}} = \frac{1}{\sqrt{2}} 
	\overline{\mathbf{f}}_{j},
	\label{eq:5.1.9paris}
\end{equation}
so that
\begin{equation}
	v=\sum_{j=1}^{m}\Big(z_{j} \frac{\partial}{\partial z_{j}} 
	+\overline{z}_{j} \frac{\partial}{\partial \overline{z}_{j}}\Big).
	\label{eq:5.1.10paris}
\end{equation}
Note that $|\frac{\partial}{\partial z_{j}}|^{2}_{g^{TX}}=|\frac{\partial}{\partial 
\overline{z}_{j}}|^{2}_{g^{TX}}=\frac{1}{2}$.
For $v, v'\in T_{x}X$, let $z, z'$ denote the corresponding 
complex coordinates. Define
\begin{equation}
\mathcal{P}_{x}(v,v')=\prod_{j=1}^{m} 
\frac{\mu_{j}(x)}{2\pi}
\exp\Big(\!-\frac{1}{4}\sum_{j=1}^{m}\mu_{j}(x)(|z_{j}|^{2}
+|z'_{j}|^{2}-2z_{j}\overline{z}_{j})\Big).
	\label{eq:5.1.11paris}
\end{equation}	
Define a weighted distance function $\Phi^{TX}_{x}(v,v')$ as follows,
\begin{equation}
\Phi^{TX}_{x}(v,v')^{2}=\sum_{j=1}^{m}\mu_{j}(x)|z_{j}-z'_{j}|^{2}.
\end{equation}
Then 
\begin{equation}
|\mathcal{P}_{x}(v,v')|=\prod_{j=1}^{m} 
\frac{\mu_{j}(x)}{2\pi}\exp\Big(\!-\frac{1}{4}\Phi^{TX}_{x}(v,v')^{2}\Big).
\end{equation}
For sufficiently small $\delta_{0}>0$, we identify the small open 
ball
$B^{X}(x,2\delta_{0})$ in $X$ with the ball $B^{T_{x}X}(0, 
2\delta_{0})$ in $T_{x}X$ via the geodesic coordinate. 	
Let $\mathrm{dist}(\cdot,\cdot)$ denote the Riemannian distance 
of $(X, g^{TX})$. There exists $C_{2}>0$ such that for $v,v'\in 
B^{T_{x}X}(0, 2\delta_{0})$, we 
have
\begin{equation}
	C_{2}\mathrm{dist}(\exp_{x}(v),\exp_{x}(v'))\geq \Phi^{TX}_{x}(v,v')\geq 
	\frac{1}{C_{2}}\mathrm{dist}(\exp_{x}(v),\exp_{x}(v')).
	\label{eq:5.1.15DLM}
\end{equation}
In particular,
\begin{equation}
	\Phi^{TX}_{x}(0,v)\geq \varepsilon_{1}^{1/2} 
	\mathrm{dist}(x,\exp_{x}(v)).
	\label{eq:5.1.15paris}
\end{equation}
Moreover, if we consider a compact subset $K\subset X$, the 
constants $\delta_{0}$ and $C_{1}$ can be chosen uniformly for all 
$x\in K$.

We trivialize the line bundle $L$ on $B^{T_{x}X}(0, 2\delta_{0})$ 
using the parallel transport with respect to $\nabla^{L}$ along 
the curve $[0,1]\ni t\mapsto tv$, $v\in B^{T_{x}X}(0, 
2\delta_{0})$. Under this trivialization, for $v,v' \in B^{T_{x}X}(0, 
2\delta_{0})$, 
\begin{equation}
	B_{p}(\exp_{x}(v),\exp_{x}(v'))\in\mathrm{End}(L_{x})=\C.
\end{equation}

By \cite[Theorems 4.2.1 \& 
6.1.1]{MM07}, for any compact subset $K\subset X$, and for any $N\in \mathbb{N}$, there 
exists $\delta>0$ and constants $C,C'>0$ such that for 
$x\in K$, $v,v'\in T_{x}X$, 
$\Vert v\Vert,\,\Vert v'\Vert \leq 2\delta$, instead of \eqref{eq:2.2.2}, we have
\begin{equation}
	\begin{split}
		&\bigg|\frac{1}{p^{m}}B_{p}(\exp_{x}(v),\exp_{x}(v'))-\sum_{r=0}^{N}\mathcal{F}_{r}(\sqrt{p}v,\sqrt{p}v')\kappa^{-1/2}(v)\kappa^{-1/2}(v')p^{-r/2}\bigg|\\
		&\leq 
		Cp^{-(N+1)/2}(1+\sqrt{p}\Vert v\Vert+\sqrt{p}\Vert 
		v'\Vert)^{2(N+m)+4}\exp(-C'\sqrt{p}\Vert v-v'\Vert)+\mathcal{O}(p^{-\infty}).
	\end{split}
	\label{eq:5.2.2paris}
\end{equation}
The functions $\mathcal{F}_{r}$, $r\in\bN$ are given as follows,
\begin{equation}
	\mathcal{F}_{r}(v,v')=\mathcal{P}_{x}(v,v')\mathcal{J}_{r}(v,v'),
	\label{eq:5.1.19DLM}
\end{equation}
where $\mathcal{J}_{r}(v,v')$ is a polynomial in $v,v'$ of degree $\leq 3r$, whose 
coefficients are smooth in $x\in X$. In particular,
\begin{equation}
	\mathcal{J}_{0}=1.
\end{equation}

The normalized Bergman kernel $P_{p}(x,y)$ for $x,y\in X$ is defined as in 
\eqref{eq:1.3.1}.
Then by exactly the same 
arguments in Subsection \ref{section2.3DLM}, we get a version of Theorem 
\ref{thm:1.3.1} for $P_{p}$ as follows.
\begin{theorem}\label{thm:5.1.1}
	Let $U$ be a relatively compact open subset of $X$, then the 
	following uniform estimates on the normalized Bergman kernel hold 
	for $x,y\in U$:
	fix $k\geq 1$ and $b>\sqrt{16k/\varepsilon_{1}}$, then we have 
	\begin{equation}
		P_{p}(x,y)=\begin{cases} 
		&(1+o(1))\exp(-\frac{p}{4}\Phi_{x}(0,v')^{2}),  \\
		&\qquad\qquad\qquad\text{uniformly for } 
		\mathrm{dist}(x,y=\exp_{x}(v'))=\Vert v'\Vert  \leq b\sqrt{\frac{\log p}{p}}, \\
		&\mathcal{O}(p^{-k}),  \qquad\;\;   \text{uniformly for } 
		\mathrm{dist}(x,y) \geq 
		b\sqrt{\frac{\log p}{p}}.
	\end{cases}
	\label{eq:5.2.4}
\end{equation}
\end{theorem}

Now we start to give the proofs to the theorems in Subsection 
\ref{section1.6}. Most of the arguments are exactly the same as given 
in Section \ref{section3DLM}.

\begin{proof}[Proof of Theorem \ref{thm:1.6.1}]
Since $U$ is relatively compact in $X$, 
the Bergman kernel $B_{p}(x,x)$, $x\in U$, 
is uniformly bonded above by $C_{\overline{U}}\,p^{m}$
by \eqref{eq:5.1.4paris}. 
By \eqref{eq:5.1.2DLM}, $d_{p}\simeq p^{m}$. The proof of 
	\eqref{eq:1.6.7paris} follows exactly by the arguments in the 
	proof of Proposition \ref{prop:5.2.2}. In particular, we have for 
	$p\gg 0$,
	\begin{equation}
		\mathbb{E}[|\mathcal{M}^{U}_{p}(s_{p})|^{d_{p}}]\lesssim 
		d_{p}^{2d_{p}+3/2}\leq p^{Cp^{m}},
		\label{eq:5.1.20DLM}
	\end{equation}
	where the constant $C>0$ is sufficiently large.

	Now we prove \eqref{eq:1.6.8paris}. At first, by \eqref{eq:5.1.20DLM}, as 
	in Corollary \ref{cor:3.1.6}, there exists a constant $C>0$, such that for any sequence 
	$\{\lambda_{p}\}_{p\in \mathbb{N}}$  of strictly positive 
	numbers, we have
	\begin{equation}
		\begin{split}
			\Upsilon_{p}\Big(\big\{s_{p}\;:\; 
			\mathcal{M}^{U}_{p}(s_{p})\geq \lambda_{p}\big\}\Big)\lesssim 
			e^{-d_{p}\log\lambda_{p}+Cd_{p}\log p}.
		\end{split}
		\label{eq:5.1.21DLM}
	\end{equation}
	
	Secondly, we apply the same 
	arguments in Subsection \ref{section3.3DLM} by taking the lattice 
	points 
	$x^{p}_{v}$, $v\in \Gamma_{p}\subset \Z^{m}$, near a fixed point 
	$x_{0}\in U$, where we identify 
	$(T_{x_{0}}X,g^{T_{x_{0}}X})$ with $\R^{m}$. Note that $n:=\sharp 
	\Gamma_{p}\simeq d_{p}$. As in 
	\eqref{eq:53.2.10}
	\begin{equation}
		\frac{d^{2}}{4p}\Vert u-v\Vert ^{2}\leq\mathrm{dist}(x^{p}_{u},x^{p}_{v})^{2}\leq \frac{4d^{2}}{p}\Vert u-v\Vert ^{2}.
		\label{eq:5.1.24DLM}
	\end{equation}
	Let $\xi(u,v)\in T_{x^{p}_{u}}X$ be the unique vector (with small 
	norm) such that 
	$\exp_{x^{p}_{u}}(\xi(u,v))=x^{p}_{v}$.
	By \eqref{eq:5.1.15paris}, \eqref{eq:5.1.24DLM},
	\begin{equation}
		\Phi_{x^{p}_{u}}(0,\xi(u,v))^{2} \geq 
		\frac{\varepsilon_{1}d^{2}}{4p}\Vert u-v\Vert ^{2}.
	\end{equation}
	This is an analog of \eqref{eq:53.2.24}. Then using instead 
	Theorem \ref{thm:5.1.1} and proceeding as in \eqref{eq:53.2.17} - \eqref{eq:3.3.37DLM}, we get that, for a sequence 
	$\{\lambda_{p}\}_{p\in \mathbb{N}}$  of positive 
	numbers less than $1$, there exist constants $C>0$, 
	$C'>0$ such that for $\forall\, p\gg 0$,
	\begin{equation}
		\Upsilon_{p}\Big(\big\{s_{p}\;:\; \mathcal{M}^{U}_{p}(s_{p})\leq 
		\lambda_{p}\big\}\Big)\leq e^{Cd_{p}\log \lambda_{p} +C'd_{p}\log 
		p}.
		\label{eq:5.1.25DLM}
	\end{equation}	
	Taking $\lambda_{p}=e^{\delta p}$ in \eqref{eq:5.1.21DLM} and 
	$\lambda_{p}=e^{-\delta p}$ in \eqref{eq:5.1.25DLM}, we get  
	\eqref{eq:1.6.8paris} upon using $d_{p}\simeq p^{m}$. This completes our proof.
\end{proof}

\begin{proof}[Proof of Theorem \ref{thm:6.2}]
	The first part of this theorem is an analog of 
	Theorem \ref{thm:1.3.5}, and its proof will follow the arguments as 
	explained in Subsections \ref{subs3.3} \& \ref{subs3.4}, by using 
	instead Theorem \ref{thm:1.6.1}. Recall that $\mathrm{dV}=\frac{\omega^{m}}{m!}$ is the volume element on $X$.
	
	Indeed, the sketched proof in Subsection \ref{subs3.3} 
	(cf.\cite[Subsection 4.1]{SZZ}) also proves that, for the relative compact 
	open subset $U$, and for any $\delta>0$, 
	there exists $C_{U,\delta}>0$ such that
	\begin{equation}
		\Upsilon_{p}\left(\Big\{s_{p}\;:\; 
		\int_{U}\big|\log{|s_{p}|_{h^{p}}}\big|\mathrm{dV}\geq 
		\delta p\Big\}\right)\leq e^{-C_{U,\delta}p^{m+1}}, \,\forall\, p\gg 0.
		\label{eq:5.1.27paris}
	\end{equation}	
To get \eqref{eq:6.0.7}, we apply the Poincar\'{e}-Lelong formula
for $\varphi\in \Omega^{m-1,m-1}_{0}(U)$,
\begin{equation}
([\mathrm{Div}(s_{p})]-pc_{1}(h,L),\varphi)=
\frac{i}{\pi}\int_{U}\log|s_{p}|_{h_{p}}\partial\overline{\partial}\varphi.
	\end{equation}
Then 
\begin{equation}
\begin{split}	
\left|\left(\frac{1}{p}[\mathrm{Div}(s_{p})]-c_{1}(h,L),\varphi\right)\right|&
\leq\frac{1}{\pi p}\left|\int_{U}\log|s_{p}|_{h_{p}}
\partial\overline{\partial}\varphi\right|\\
&\leq \frac{1}{\pi p}\sup_{x\in X}\left|
\frac{\partial\overline{\partial}\varphi(x)}{\mathrm{dV}(x)}\right|
\cdot\int_{U}\big|\log{|s_{p}|_{h^{p}}}\big|\,\mathrm{dV}.
		\end{split}
		\label{eq:5.1.29paris}
	\end{equation}
	Then by \eqref{eq:5.1.27paris}, we get \eqref{eq:6.0.7}. 
	As a consequence of \eqref{eq:6.0.7}, 
	the proof to \eqref{eq:1.6.2DLM} follows exactly from the same 
	arguments as in the part (a) of Subsection \ref{subs3.5}. This 
	completes our proof.
\end{proof}

\begin{proof}[Proof of Theorem \ref{thm:6.3}]
	We only need to prove \eqref{eq:6.1.10}, and 
	\eqref{eq:1.6.14paris} is just its direct consequence. Due to 
	results in Theorem \ref{thm:6.2} and that $U$ is relatively 
	compact, the Proof of this theorem is quite routine as in 
	Subsection \ref{subs3.5} (also in \cite[Subsection 4.2.2]{SZZ}).
	
	Let $\delta>0$ be arbitrary, and we take $\psi_{1}$, $\psi_{2}\in 
\mathcal{C}^{\infty}_{0}(X,\R_{\geq 0})$ such that
\begin{equation}
	\begin{split}
		&0\leq \psi_{1}\leq\chi_{U}\leq \psi_{2}\leq 1,\\
		&\int_{X} \psi_{1}\frac{c_{1}(L,h)^{m}}{m!}\geq 
		\mathrm{Vol}^{L}_{2m}(U)-\delta,\\
		&\int_{X} \psi_{2}\frac{c_{1}(L,h)^{m}}{m!}\leq 
		\mathrm{Vol}^{L}_{2m}(U)+\delta.
	\end{split}
	\label{eq:5.1.30paris}
\end{equation}
Set $\varphi_{j}=\frac{\psi_{j}}{(m-1)!}c_{1}(L,h)^{m-1}$, $j=1,2$. 
Then we apply Theorem \ref{thm:6.2} to $\varphi_{j}$ separately, we 
get exactly \eqref{eq:6.1.10}. Our proof is completed.
\end{proof}
\begin{remark}
	Assume that $\{\Upsilon_{p}\}_{p\in\bN}$ is defined as in 
	Example \ref{eg:1.3.2} with $\sigma_{p}=1$. Then, similar to Proposition \ref{prop:1.2.3} and the second part of 
	\cite[Theorem 1.4]{SZZ}, we can also give a lower bound ($\simeq 
	e^{-C_{U}\,p^{m+1}}$) for the 
	hole probabilities for a relative compact nonempty open subset 
	$U\subset X$, provided there exists a nowhere vanishing section 
	on $\overline{U}$.
\end{remark}

We give now exhibit two classes of manifolds for which Theorem 
\ref{thm:6.2} applies, each of them has its own interests in various 
fields of complex geometry.
\begin{example}
	Let $M$ is a compact complex manifold of dimension $m$, 
	$\Sigma$ is an analytic subvariety of $M$, $X:=M\setminus\Sigma$.
	We assume that $X$ admits a complete K\"ahler metric $\omega$ 
	such that $\ric_\omega\leq-\lambda\omega$,
	for some constant $\lambda>0$.
	Assume moreover that $\dim\Sigma\leq m-k$, $k\geq2$.
	Then $H^0_{(2)}(X,K_X^p)\subset H^0(M,K_M^p)$
	and $d_p=\dim H^0_{(2)}(X,K_X^p)=\mathcal{O}(p^m)$ as $p\to\infty$.

\end{example}
\begin{example}
	Let $D$ be a bounded symmetric domain in $\mathbb{C}^m$ and let $\Gamma$ be a
	neat arithmetic group acting properly discontinuously on $D$ (see \cite[p.\ 253]{Mum77}).
	Then $U:=D/\Gamma$ is a smooth quasi-projective variety, called an arithmetic variety.
	By \cite{AMRT:10}, $U$ admits a smooth toroidal compactification $X$.
	In particular, $\Sigma:=X\setminus U$ is a divisor with normal crossings.
	The Bergman metric $\omega^{\mathcal B}_{D}$ on $D$ descends to a complete
	K\"ahler metric $\omega:=\omega^{\mathcal B}_{U}$ on $U$.
	Moreover, $\omega$ is K\"ahler-Einstein with $\ric_{\omega}=-\omega$
	(since the metric $\omega^{\mathcal B}_{D}$ has this property).
	We denote by $h^{K_U}$ the Hermitian metric induced by $\omega$ on $K_U$.
	We wish to study the spaces $H^0_{(2)}(U,K_U^p)$ of $\mathcal{L}^{2}$-pluricanonical
	sections with respect to the metric $h^{K^p_U}$ and the volume form 
	$\omega^m$.
\end{example}

\bibliographystyle{plain}

\end{document}